\documentclass[a4paper,11pt,twoside]{amsart}

\usepackage{graphicx}
\usepackage{amsfonts}
\usepackage{amsmath}
\usepackage{amsthm}
\usepackage{amssymb}
\usepackage[all]{xy}
\usepackage{hyperref}
\usepackage{aliascnt}

\newtheorem{thm}{Theorem}[section]

\newtheorem{thmInt}{Theorem}[section]

\newaliascnt{prop}{thm}
\newtheorem{prop}[prop]{Proposition}
\aliascntresetthe{prop}

\newaliascnt{lem}{thm}
\newtheorem{lem}[lem]{Lemma}
\aliascntresetthe{lem}

\newaliascnt{cor}{thm}
\newtheorem{cor}[cor]{Corollary}
\aliascntresetthe{cor}

\theoremstyle{definition}

\newaliascnt{definition}{thm}
\newtheorem{definition}[definition]{Definition}
\aliascntresetthe{definition}

\newaliascnt{remark}{thm}
\newtheorem{remark}[remark]{Remark}
\aliascntresetthe{remark}

\newaliascnt{ex}{thm}
\newtheorem{ex}[ex]{Example}
\aliascntresetthe{ex}

\numberwithin{equation}{section}

\DeclareMathOperator{\im}{im} 
\DeclareMathOperator{\cok}{coker} 
\DeclareMathOperator{\spec}{Spec} 
\newcommand{\iso}{\cong}
\newcommand{\niso}{\ncong}

\newcommand{\farg}{-} 
\newcommand{\id}{\mathrm{id}}
\newcommand{\st}{:} 
\newcommand{\comp}{\circ} 
\newcommand{\mor}[1]{\xrightarrow{#1}}
\newcommand{\mono}{\hookrightarrow} 
\newcommand{\epi}{\twoheadrightarrow} 
\newcommand{\card}[1]{\lvert#1\rvert} 
\newcommand{\rest}[1]{|_{#1}} 

\newcommand{\K}{\Bbbk} 
\newcommand{\cat}[1]{{\mathbf{#1}}} 
\newcommand{\opp}{^{\circ}} 

\newcommand{\Ext}{\mathrm{Ext}}

\newcommand{\Br}{\mathrm{Br}}
\newcommand{\Eq}[1][]{\mathbf{Eq}^{\mathrm{#1}}}
\newcommand{\sh}[2][1]{#2[#1]} 
\newcommand{\FM}[2][]{\Phi^{#1}_{#2}} 
\newcommand{\D}[1][]{\mathrm{D}^{#1}} 
\newcommand{\Db}{\D[b]} 
\newcommand{\Dp}[1][]{\cat{Perf}_{#1}} 
\newcommand{\Dq}{\mathrm{D}_{\mathrm{qc}}} 
\newcommand{\dgD}{\mathcal{D}} 
\newcommand{\Perf}{\mathrm{Perf}^{\,\mathrm{dg}}} 
\newcommand{\rd}{\mathbf{R}} 
\newcommand{\ld}{\mathbf{L}} 
\newcommand{\lotimes}{\overset{\ld}{\otimes}} 
\newcommand{\fun}[1]{\mathsf{#1}} 
\newcommand{\Mod}[1]{\mathrm{Mod}(#1)} 
\newcommand{\dgMod}[1]{\mathrm{dgMod}(#1)}
\newcommand{\SF}[1]{\mathrm{SF}(#1)} 
\newcommand{\SFa}[1]{\mathrm{SF}_\alpha(#1)} 
\newcommand{\Coh}{\cat{Coh}}
\newcommand{\Qcoh}{\cat{Qcoh}}
\newcommand{\lto}{\longrightarrow}
\newcommand{\Ob}{\mathrm{Ob}}

\newcommand{\hocolim}{\mathrm{hocolim}\,}

\newcommand{\ZZ}{\mathbb{Z}}

\newcommand{\Ind}{\fun{Ind}}
\newcommand{\Res}{\fun{Res}}
\newcommand{\Hqe}{\cat{Hqe}} 

\newcommand{\Plus}{\coprod}

\newcommand{\ke}{\mathcal{E}}
\newcommand{\cA}{\cat{A}}
\newcommand{\cB}{\cat{B}}
\newcommand{\cC}{\cat{C}}
\newcommand{\cD}{\cat{D}}
\newcommand{\cG}{\cat{G}}
\newcommand{\cL}{\cat{L}}
\newcommand{\cM}{\cat{M}}
\newcommand{\cN}{\cat{N}}
\newcommand{\cR}{\cat{R}}
\newcommand{\cS}{\cat{S}}
\newcommand{\cT}{\cat{T}}
\newcommand{\fE}{\fun{E}}
\newcommand{\fF}{\fun{F}}
\newcommand{\fG}{\fun{G}}
\newcommand{\fH}{\fun{H}}
\newcommand{\fI}{\fun{I}}

\newcommand{\fM}{\fun{M}}
\newcommand{\fPo}{\fun{Y}'}

\newcommand{\ko}{\mathcal{O}}

\newcommand{\fS}{\fun{S}}
\newcommand{\fT}{\fun{T}}
\newcommand{\fQ}{\fun{Q}}
\newcommand{\Ho}{\mathrm{H}^0}
\newcommand{\Yon}[1][\cA]{\fun{Y}^{#1}} 
\newcommand{\YA}[1]{\Yon[\cA](#1)}
\newcommand{\dgYon}[1][\cA]{\fun{Y}^{#1}_{\mathrm{dg}}} 
\newcommand{\comma}[1]{(\Yon\downarrow#1)} 
\newcommand{\fgt}[1]{\fun{F}_{#1}} 
\newcommand{\limsrc}[1]{X_{#1}} 
\newcommand{\limtar}[1]{Y_{#1}} 
\newcommand{\limmap}[1]{\alpha_{#1}} 
\DeclareMathOperator{\dlim}{\underset{\lto}{lim}} 
\newcommand{\Mor}{\mathrm{Mor}}
\newcommand{\inc}[1]{\iota_{#1}} 
\newcommand{\pro}[1]{\rho_{#1}} 
\newcommand{\htp}{\sim} 
\newcommand{\MOD}[1]{\mathrm{MOD}\text{-}#1} 
 
\newcommand{\Ac}{\mathrm{Ac}} 

\begin{document}

	\title[Uniqueness of dg enhancements for Grothendieck categories]{Uniqueness of dg enhancements for the derived category of a Grothendieck category}

	\author{Alberto Canonaco and Paolo Stellari}

	\address{A.C.: Dipartimento di Matematica ``F. Casorati'', Universit{\`a}
	degli Studi di Pavia, Via Ferrata 5, 27100 Pavia, Italy}
	\email{alberto.canonaco@unipv.it}

	\address{P.S.: Dipartimento di Matematica ``F.
	Enriques'', Universit{\`a} degli Studi di Milano, Via Cesare Saldini
	50, 20133 Milano, Italy}
	\email{paolo.stellari@unimi.it}
    \urladdr{\url{http://users.unimi.it/stellari}}
	
	\thanks{A.~C.~ was partially supported by the national research project
	  ``Spazi di Moduli e Teoria di Lie'' (PRIN 2012).
	P.~S.~ is partially supported by the grants FIRB 2012 ``Moduli Spaces and Their Applications'' and
	the national research project ``Geometria delle Variet\`a Proiettive'' (PRIN 2010-11).}



\begin{abstract}
We prove that the derived category of a Grothendieck abelian category has a unique dg enhancement. Under some additional assumptions, we show that the same result holds true for its subcategory of compact objects. As a consequence, we deduce that the unbounded derived category of quasi-coherent sheaves on an algebraic stack and  the category of perfect complexes on a noetherian concentrated algebraic stack with quasi-finite affine diagonal and enough perfect coherent sheaves have a unique dg enhancement. In particular, the category of perfect complexes on a noetherian scheme with enough locally free sheaves has a unique dg enhancement.
\end{abstract}

\maketitle

\setcounter{tocdepth}{1}
\tableofcontents

\section*{Introduction}

The relation between triangulated categories and higher categorical structures is highly non-trivial, very rich in nature and with various appearances in the recent developments of derived algebraic geometry. The easiest thing we can do is to produce a triangulated category $\cT$ out of a pretriangulated dg category $\cC$ by taking the homotopy category of $\cC$. Roughly speaking, a pretriangulated dg category $\cC$ whose homotopy category is equivalent to a triangulated category $\cT$ is called a \emph{dg enhancement} (or \emph{enhancement}, for short) of $\cT$.

Now, there exist triangulated categories with no
enhancements at all. For example, this happens to some triangulated
categories naturally arising in topology (see \cite{S} or \cite[Section 3.6]{K} for a discussion about this). Triangulated categories
admitting an enhancement are called \emph{algebraic} (as it is explained, for example, in \cite[Section 3]{S}, algebraic triangulated categories are
often defined in other equivalent ways). In practice, all triangulated
categories one usually encounters in algebra or algebraic geometry are
algebraic. For instance, the derived category of a Grothendieck
category, as well as its full subcategory of compact objects are algebraic. Recall that a \emph{Grothendieck category} is an abelian category $\cG$ which is closed under small coproducts, has a small set of generators $\cS$ and the direct limits of short exact sequences are exact. The objects in $\cS$ are generators in the sense that, for any $C$ in $\cG$, there exists an epimorphism $S\epi C$ in $\cG$, where $S$ is a small coproduct of objects in $\cS$. 

In particular, if $X$ is a scheme or, more generally, an algebraic
stack it is not difficult to construct explicit enhancements of the
derived category $\D(\Qcoh(X))$ of the Grothendieck category of
quasi-coherent sheaves on $X$, of the bounded derived categories of
coherent sheaves $\Db(X)$ and of the category of perfect complexes $\Dp(X)$ on $X$. For example, this can be achieved either by taking complexes of injective sheaves or, under mild assumptions, $\mathrm{\check{C}}$ech resolutions or chain complexes of sheaves in the corresponding categories or perfect complexes (see \cite{BLL} and \cite{LS}).

Even when we know that an enhancement exists, one may wonder whether it is unique. Roughly, we say that a triangulated category $\cT$ has a \emph{unique enhancement} $\cC$ if any other enhancement is related to $\cC$ by a sequence of quasi-equivalences. These are the analogue, at the dg level, of the exact equivalences in the triangulated setting.
Actually, at this level of generality, we may not expect a positive
answer to the above question. Indeed, the result of Dugger and Shipley \cite{DS}
easily yields an example of two $\ZZ$-linear pretriangulated dg categories
which are not quasi-equivalent but whose homotopy categories are
equivalent. The search of a similar example over a field rather than over a commutative ring is still a challenge.

Again, if we move to the geometric setting, then for a long while it was expected that any of the three triangulated categories $\D(\Qcoh(X))$, $\Db(X)$ and $\Dp(X)$ should have unique enhancements, when $X$ is a (quasi-)projective scheme. This was formally stated as a conjecture (even in a stronger form) by Bondal, Larsen and Lunts \cite{BLL}.

As we will explain later, this conjecture was positively solved by Lunts and Orlov in their seminal paper \cite{LO}. It should be noted that the quest for uniqueness of enhancements has a foundational relevance that cannot be overestimated by the `working algebraic geometer'. Let us just mention an instance where the fact of having a unique enhancement has interesting consequences. The homological version of the so called \emph{Mirror Symmetry Conjecture} by Kontsevich \cite{Ko} predicts the existence of an $A_\infty$-equivalence between a dg enhancement of $\Db(X)$, for $X$ a smooth projective scheme, and the Fukaya category of the mirror $Y$ of $X$, which is actually an $A_\infty$-category. The fact that the dg enhancements are unique allows us to conclude that finding an $A_\infty$-equivalence (or rather a sequence of them) is the same as finding an exact equivalence between the corresponding homotopy categories. More generally, several geometric problems can be lifted to the dg level and treated there in a universal way (e.g.\ moduli problems or the characterization of exact functors). Having bridges between the different dg incarnations of the same triangulated or geometric problem is then crucial.

\medskip

Let us now explain the contributions of this paper to the problem of showing the uniqueness of dg enhancements in geometric settings. The first point to make, which should be clear from now on, is that the analysis of these questions about $\D(\Qcoh(X))$ or $\Dp(X)$ (or $\Db(X)$) follows slightly different paths. In particular, they can be deduced from two different general criteria whose statements are similar but whose proofs are rather different in nature.

We first consider the case of $\D(\Qcoh(X))$ and, setting the problem at a more abstract level, we first prove the following general result.

\begin{thmInt}\label{thm:main1}
If $\cG$ is a Grothendieck category, then $\D(\cG)$ has a unique enhancement.
\end{thmInt}

We will explain later some key features in the proof. For the moment, we just recall that the main geometric applications are the following:
\begin{itemize}
\smallskip
\item If $X$ is an algebraic stack, $\D(\Qcoh(X))$ has a unique enhancement (see \autoref{cor:geo1});
\smallskip
\item If $X$ is scheme and $\alpha$ is an element in the Brauer group $\Br(X)$ of $X$, then the twisted derived category $\D(\Qcoh(X,\alpha))$ has a unique enhancement (see \autoref{cor:geo2}).
\end{itemize}

\smallskip

Now, if we want to study the enhancements of $\Dp(X)$ (and, consequently, of $\Db(X)$), we should keep in mind that if $X$ is a quasi-compact and semi-separated scheme, then $\Dp(X)$ is the triangulated subcategory of $\D(\Qcoh(X))$ consisting of compact objects. Our general result in this direction is then the following.

\begin{thmInt}\label{thm:crit2}
Let $\cG$ be a Grothendieck category with a small set $\cA$ of generators such that
\begin{enumerate}
\item $\cA$ is closed under finite coproducts;
\item Every object of $\cA$ is a noetherian object in $\cG$;
\item If $f\colon A'\epi A$ is an epimorphism of $\cG$ with
$A,A'\in\cA$, then $\ker f\in\cA$;
\item For every $A\in\cA$ there exists $N(A)>0$ such that
$\D(\cG)\left(A,\sh[N(A)]{A'}\right)=0$ for every $A'\in\cA$.
\end{enumerate}
Then $\D(\cG)^c$ has a unique enhancement.
\end{thmInt}

Here $\D(\cG)^c$ denotes the subcategory of compact objects in $\D(\cG)$. One may wonder why the result above is conditional while \autoref{thm:main1} does not include any specific assumptions on $\cG$. We will try to explain later that this is, in a sense, unavoidable but to reassure the reader about the mildness of (1)--(4), let us now discuss some geometric cases where \autoref{thm:crit2} applies:
\begin{itemize}
\smallskip
\item If $X$ is a noetherian concentrated algebraic stack with quasi-finite affine diagonal and with enough perfect coherent sheaves, then $\Dp(X)$ has a unique enhancement (see \autoref{prop:geocomp});
\smallskip
\item As a special (but maybe easier to understand) instance of the above case, we have that if $X$ is a noetherian scheme with enough locally free sheaves, then $\Dp(X)$ has a unique enhancement (see \autoref{cor:geocomp});
\smallskip
\item Under the same assumptions on the scheme $X$, the category $\Db(X)$ has a unique enhancement (see \autoref{cor:geocomp2}).
\end{itemize}
\smallskip
The (more or less) standard terminology involved in the above statements will be briefly recalled in \autoref{subsect:compex}.

It is very likely that \autoref{thm:main1} and \autoref{thm:crit2} may be used in other geometric contexts. One direct application of the circle of ideas appearing in the proofs of these two results concerns the existence of exact equivalences. In particular, if $X_1$ and $X_2$ are noetherian schemes with enough locally free sheaves then the set of equivalences between $\Dp(X_1)$ and $\Dp(X_2)$ is not empty if and only if the same is true for the set of equivalences between $\D(\Qcoh(X_1))$ and $\D(\Qcoh(X_2))$. This is \autoref{prop:FM}.

\subsection*{The strategy of the proof}

Before entering into some details of the proof it is worth pointing out the general approach to \autoref{thm:main1} and \autoref{thm:crit2}. Even if these results have a dg flavour, the idea is to reduce them to questions about Verdier quotients of triangulated categories. Unfortunately, some of these latter questions are highly non-trivial and involve deep problems concerning the description of the subcategory of compact objects of a quotient. This is the reason why our proofs, which are conceptually quite simple, become technically rather involved.

\smallskip

Let us try to make this more precise and consider first
\autoref{thm:main1}. The key observation is that the derived category
$\D(\cG)$ of a Grothendieck category $\cG$ is \emph{well generated} in the
sense of Neeman \cite{N2}. Thus one can choose a small set $\cA$ of generators for $\cG$ such that $\D(\cG)$ is naturally equivalent to the quotient $\dgD(\cA)/\cL$, where $\dgD(\cA)$ is the derived category of $\cA$, seen as a dg category, and $\cL$ is an appropriate localizing subcategory of $\dgD(\cA)$.

This is carried out in \autoref{subsect:Grothabstr}, where we also explain that \autoref{thm:main1} follows easily once we prove the following general criterion.

\begin{thmInt}\label{thm:crit1}
Let $\cA$ be a small category considered as a dg category concentrated in degree $0$ and let $\cL$ be a localizing subcategory of $\dgD(\cA)$ such that:
\begin{itemize}
\item[{\rm (a)}] The quotient $\dgD(\cA)/\cL$ is a well generated triangulated category;
\item[{\rm (b)}] The quotient functor $\fQ\colon\dgD(\cA)\to\dgD(\cA)/\cL$ is right vanishing.
\end{itemize}
Then $\dgD(\cA)/\cL$ has a unique enhancement.
\end{thmInt}

We will give the precise definitions of well generated triangulated category and right vanishing functor, respectively, in \autoref{def:wellgen} and \autoref{abc}. For the moment, it is enough to keep in mind that every compactly generated triangulated category is well generated. Moreover, as explained later in \autoref{ex:rightvancompact}, $\fQ$ is right vanishing if it satisfies the following conditions:
\begin{enumerate}
\item[(b.1)] $\fQ(\Yon(A))$ is compact in $\dgD(\cA)/\cL$, for every $A\in\cA$;
\item[(b.2)] $\dgD(\cA)/\cL\left(\fQ(\Yon(A)),\fQ(\Yon(A'))[k]\right)=0$, for every $A,A'\in\cA$ and every integer $k<0$.
\end{enumerate}
Here $\Yon\colon\cA\to\dgD(\cA)$ is the Yoneda functor, while $\dgD(\cA)/\cL\left(\farg,\farg\right)$ denotes the Hom-space in the category $\dgD(\cA)/\cL$.

The idea of the proof, which occupies the whole of \autoref{sect:firstcrit}, is very much inspired by the proof of \cite[Theorem 2.7]{LO} but it differs at some technical steps. We will try to clarify them in a while when comparing our results to those in \cite{LO}. The geometric applications mentioned above and discussed in \autoref{subsect:Grothex}, can be deduced easily from the fact that, in all those cases, the category of quasi-coherent sheaves is a Grothendieck category, under our assumptions on $X$.

\medskip

Once we have the equivalence $\D(\cG)\iso\dgD(\cA)/\cL$ as above, it is clear that to prove \autoref{thm:crit2} we have to show that the triangulated subcategory $(\dgD(\cA)/\cL)^c$ of compact objects in $\dgD(\cA)/\cL$ has a unique enhancement, for a smart choice of $\cA$ (see \autoref{subsect:firstred}). For this, one may hope to use \autoref{thm:LO} which was proved by Lunts and Orlov in \cite{LO}. Indeed, this criterion for uniqueness asserts that if we change (a) in \autoref{thm:crit1} to
\begin{itemize}
\item[(a')] $\cL^c=\cL\cap\dgD(\cA)^c$ and $\cL$ is generated by $\cL^c$,
\end{itemize}
and we keep (b), then we can deduce that $(\dgD(\cA)/\cL)^c$ has a unique enhancement as well. In this case, the proof is not too difficult, as we can use the fact that $(\dgD(\cA)/\cL)^c$ and $\dgD(\cA)^c/\cL^c$  are nicely related, as explained in \cite{N1} (see \autoref{thm:NeemanComp}).

The issue here is that (a') is not easily verified. Indeed, if $\cL$ satisfies (a'), then the inclusion functor $\cL\hookrightarrow\cT$ has a right adjoint which preserves small coproducts. In general, given a compactly generated triangulated category $\cT$ closed under small coproducts, a localizing subcategory $\cL$ of $\cT$ such that the inclusion $\cL\hookrightarrow\cT$ has the above property is called a \emph{smashing subcategory}.

For a while, it was conjectured that all smashing subcategories $\cL$ of a triangulated category $\cT$ as above should verify (a'). This goes under the name of \emph{Telescope Conjecture} (see \cite[1.33]{Ra} and \cite[3.4]{Bou}). 
Unfortunately, the Telescope Conjecture is known to be false in this generality \cite{KTel} and to be true in very few examples (see, for example, \cite{NTel}). This shows that we cannot expect that (a') holds true in general or easily.

In view of this discussion, the main task which is carried out in \autoref{subsect:assumpt} is to identify the correct choice for $\cA$ such that (a') holds for the corresponding localizing subcategory $\cL$. To get this, one has to impose some additional assumptions on $\cG$ and $\cA$. This is the reason why the hypotheses (1)--(4) appear in \autoref{thm:crit2}. Assuming this, the proof of \autoref{thm:crit2} is contained in \autoref{subsect:assumpt} and the core of the argument is then \autoref{cptgen}.

The applications concerning the uniqueness of enhancements for $\Dp(X)$ (see \autoref{prop:geocomp} and \autoref{cor:geocomp}) and $\Db(X)$ (see \autoref{cor:geocomp2}) are rather easy consequences once \autoref{thm:crit2} and, more precisely, \autoref{cptgen} are established.

\subsection*{Related work}

As we recalled before, in \cite{BLL} Bondal, Larsen and Lunts first conjectured that all enhancements of $\Db(X)$, for $X$ a smooth projective scheme, should be unique. In the same paper, they show that all `standard' enhancements are related by quasi-equivalences, giving the first evidence to their conjecture.

After that, the main reference is \cite{LO} which is certainly the principal source of inspiration for this paper as well. Let us briefly summarize the results contained in that paper and compare them to ours. For $\cA$ a small category as in \autoref{thm:crit1}, Lunts and Orlov show that $\dgD(\cA)/\cL$ has a unique enhancement if $\fQ$ has a right adjoint and (b.1) and (b.2) above hold true.
This is \cite[Theorem 2.7]{LO}. It should be noted that the existence of the right adjoint to $\fQ$ and (b.1) together imply that $\dgD(\cA)/\cL$ is compactly generated. This is a special instance of our assumption (a) in \autoref{thm:crit1}. Moreover, by \cite{N3}, there are examples of Grothendieck categories whose derived category is not compactly generated but is well generated. Hence \autoref{thm:crit1} is certainly a generalization of \cite[Theorem 2.7]{LO}. The geometric consequences of \cite[Theorem 2.7]{LO}, which are discussed in the same paper, are then:
\begin{itemize}
\smallskip
\item For a Grothendieck category $\cG$, the derived category $\D(\cG)$ has a unique enhancement, if $\cG$ has a small set of generators which are compact in $\D(\cG)$ (see \cite[Theorem 2.9]{LO});
\smallskip
\item This implies that if $X$ is a quasi-compact and separated scheme that has enough locally free sheaves, then $\D(\Qcoh(X))$ has a unique enhancement (see \cite[Theorem 2.10]{LO}).
\end{itemize}
\smallskip

As a second step, Lunts and Orlov deduce from \cite[Theorem 2.8]{LO} that if $X$ is a quasi-projective scheme, then both $\Dp(X)$ and $\Db(X)$ have unique enhancements. A strong version of uniqueness is then discussed. Namely, they prove that these two categories have strongly unique enhancements when $X$ is projective and another technical assumption is satisfied. This is out of the scope of this paper but we believe that the techniques discussed here might have applications to show the strong uniqueness of dg enhancements in new cases. Indeed, \autoref{cptgen} has already been applied to prove the strong uniqueness of the category of perfect supported complexes (see \cite[Theorem 1.2]{CS}).

New interesting enhancements of geometric nature have been recently introduced by Lunts and Schn\"urer in \cite{LS}. Roughly speaking, they were used to show that the dg notion of Fourier--Mukai functor and the triangulated one agree, under some assumptions on the schemes. This important result was previously stated in \cite{To} but without a rigorous proof.

\subsection*{Plan of the paper}

This paper starts with a quick recollection of results about localizations of triangulated categories and of some properties of well generated triangulated categories (see \autoref{sect:wellgen}).

\autoref{sect:dgenhancements} and \autoref{sect:abstract} have a rather abstract nature. They cover some basic material about dg categories and dg enhancements with an emphasis on the case of enhancements of well generated triangulated categories. \autoref{sect:abstract} provides some properties of special functors which are used in the proof of \autoref{thm:crit1}.

In \autoref{sect:firstcrit} we prove \autoref{thm:crit1} while \autoref{thm:main1}, together with its geometric applications, is proved in \autoref{sect:Groth}. The proof of \autoref{thm:crit2}, of \autoref{prop:geocomp} and of \autoref{cor:geocomp} are the contents of \autoref{sect:crit2}.

\autoref{sect:applications} contain two further applications. The first one, concerning the uniqueness of enhancements for $\Db(X)$, is proved in \autoref{subsect:Db}. The second one, about Fourier--Mukai functors is explained in \autoref{subsect:FM}.

\subsection*{Notation}

All categories and functors are assumed to be $\K$-linear, for a fixed commutative ring $\K$. By a $\K$-linear category we mean a category
whose Hom-spaces are $\K$-modules and such that the compositions
are $\K$-bilinear, not assuming that finite coproducts exist.

Throughout the paper, we assume that a universe containing an infinite set is fixed. Several definitions concerning dg categories need special care because they may, in principle, require a change of universe. All possible subtle logical issues in this sense can be overcome in view of \cite[Appendix A]{LO}. A careful reader should have a look at it. After these warnings and to simplify the notation, throughout the rest of the paper we will not mention explicitly the universe we are working in, as it should be clear from the context. The members of this universe will be called small sets. For example, when we speak about small coproducts in a category, we mean coproducts indexed by a small set. If not stated otherwise, we always assume that the Hom-spaces in a category form a small set. A category is called \emph{small} if the isomorphism classes of its objects form a small set.

If $\cT$ is a triangulated category and $\cS$ a full triangulated subcategory of $\cT$, we denote by $\cT/\cS$ the Verdier quotient of $\cT$ by $\cS$. In general, $\cT/\cS$ is not a category according to our convention (namely, the Hom-spaces in $\cT/\cS$ need not be small sets), but it is in many common situations, for instance when $\cT$ is small. 

Given a category $\cC$ and two objects $C_1$ and $C_2$ in $\cC$, we denote by $\cC(C_1,C_2)$ the Hom-space between $C_1$ and $C_2$. If $\fF\colon\cC\to\cD$ is a functor and $C_1$ and $C_2$ are objects of $\cC$, then we denote by $\fF_{C_1,C_2}$ the induced map $\cC(C_1,C_2)\to\cD(\fF(C_1),\fF(C_2))$.

If $I$ is a set, $\card{I}$ denotes the cardinality of $I$.

\section{Well generated triangulated categories and localizations}\label{sect:wellgen}

In this section we use Krause's equivalent treatment (see \cite{K1}) of Neeman's notion of well generated triangulated category (see \cite{N2}). For a very clear survey about this subject, the reader can have a look at \cite{K2}.

From now on, we assume $\cT$ to be a triangulated category with small coproducts. Given a cardinal $\alpha$, an object $S$ of $\cT$ is \emph{$\alpha$-small} if every map $S\to\Plus_{i\in I}X_i$ in $\cT$ (where $I$ is a small set) factors through $\Plus_{i\in J}X_i$, for some $J\subseteq I$ with $\card{J}<\alpha$. Recall that a cardinal $\alpha$ is called \emph{regular} if it is not the sum of fewer than $\alpha$ cardinals, all of them smaller than $\alpha$.

\begin{definition}\label{def:wellgen}
The category $\cT$ is \emph{well generated} if there exists a small set $\cS$ of objects in $\cT$ satisfying the following properties:
\begin{enumerate}
\item[(G1)]\label{G1} An object $X\in\cT$ is isomorphic to $0$, if and only if $\cT(S,X[j])=0$, for all $S\in\cS$ and all $j\in\ZZ$;
\item[(G2)]\label{G2} For every small set of maps $\{X_i\to Y_i\}_{i\in I}$ in $\cT$, the induced map $\cT(S,\Plus_iX_i)\to\cT(S,\Plus_i Y_i)$ is surjective for all $S\in\cS$, if $\cT(S,X_i)\to\cT(S, Y_i)$
is surjective, for all $i\in I$ and all $S\in\cS$;
\item[(G3)]\label{G3} There exists a regular cardinal $\alpha$ such that every object of $\cS$ is $\alpha$-small.
\end{enumerate}
\end{definition}

When the category $\cT$ is well generated and we want to put emphasis on the cardinal $\alpha$ in (G3), we say that $\cT$ is \emph{$\alpha$-well generated} by the set $\cS$. In this situation, following
\cite{K1}, we denote by $\cT^\alpha$ the smallest $\alpha$-localizing
subcategory of $\cT$ containing $\cS$. Recall that a full triangulated
subcategory $\cL$ of $\cT$ is \emph{$\alpha$-localizing} if it is
closed under $\alpha$-coproducts and under direct summands (the latter condition is actually redundant if $\alpha>\aleph_0$). By definition, an $\alpha$-coproduct is a coproduct of strictly less than $\alpha$ summands. On the other hand, $\cL$ is \emph{localizing} if it is closed under small coproducts in $\cT$. The objects in $\cT^\alpha$ are called \emph{$\alpha$-compact}. Thus we will sometimes say that $\cT$ is
\emph{$\alpha$-compactly generated} by the set of \emph{$\alpha$-compact generators} $\cS$.

\begin{remark}\label{rmk:genalpha}
(i) It is easy to observe that the objects in $\cT^\alpha$ are $\alpha$-small (see, for example, \cite[Lemma 5]{K1}).

(ii) As alluded by the notation and explained in \cite{K1,N2}, the subcategory $\cT^\alpha$ does not depend on the choice of the set $\cS$ of $\alpha$-compact generators. Moreover, for any well generated triangulated category $\cT$, one has $\cT=\bigcup_{\beta}\cT^{\beta}$, where $\beta$ runs through all sufficiently large regular cardinals.

(iii) When $\alpha=\aleph_0$, then $\cT^\alpha=\cT^c$, the full
triangulated subcategory of compact objects in $\cT$. Notice that,
in this case, $\cT$ is $\aleph_0$-compactly generated by
$\cS\subseteq\cT^c$ if (G1) holds (indeed, (G3) holds by
definition of compact object, whereas (G2) is automatically
satisfied). Following the usual convention, we simply say that $\cT$ is
\emph{compactly generated} by $\cS$.
\end{remark}

\begin{ex}\label{ex:alphacompgen}
Let $\cG$ be a Grothendieck category. Then the derived category
$\D(\cG)$ is well generated (see \cite[Theorem 0.2]{N3} and \cite[Example 7.7]{K2}).
\end{ex}

Given a small set $\cS$ of objects in $\cT$, we say that \emph{$\cS$ generates $\cT$} if $\cT$ is the smallest localizing subcategory of $\cT$ containing $\cS$.

\begin{prop}[\cite{P}, Proposition 5.1]\label{prop:PortaG1}
Let $\cT$ be a well generated triangulated category. Then a small set $\cS$ of objects in $\cT$ satisfies {\rm (G1)} if and only if $\cS$ generates $\cT$.
\end{prop}

\smallskip

Let us investigate a bit more when quotients by localizing subcategories can be well generated. When $\alpha=\aleph_0$, we have the following result which we will need later.

\begin{thm}[\cite{N1}, Theorem 2.1]\label{thm:NeemanComp}
Let $\cT$ be a compactly generated triangulated category and let $\cL$ be a localizing subcategory which is generated by a small set of compact objects. Then
\begin{itemize}
\item[{\rm (i)}] $\cT/\cL$ has small Hom-sets and it is compactly generated;
\item[{\rm (ii)}] $\cL^c=\cL\cap\cT^c$;
\item[{\rm (iii)}] The quotient functor $\fQ\colon\cT\to\cT/\cL$ sends $\cT^c$ to $(\cT/\cL)^c$;
\item[{\rm (iv)}] The induced functor $\cT^c/\cL^c\to(\cT/\cL)^c$ is fully faithful and identifies $(\cT/\cL)^c$ with the idempotent completion of $\cT^c/\cL^c$.
\end{itemize}
\end{thm}

Recall that the fact that $(\cT/\cL)^c$ is the idempotent completion of $\cT^c/\cL^c$ simply means that any object in $(\cT/\cL)^c$ is isomorphic to a summand of an object in $\cT^c/\cL^c$. A similar result holds for well generated triangulated categories (see, for example,  \cite[Theorem 7.2.1]{K2}).

In general, assume that $\cT$ is well generated by a small set
$\cS$. Let $\cL$ be a localizing subcategory of $\cT$ such that the
quotient $\cT/\cL$ is well generated. Denote by
\[
\fQ\colon\cT\longrightarrow\cT/\cL
\]
the quotient functor.

\begin{remark}\label{rmk:exadj}
As we assume that $\cT/\cL$ is well generated, in particular, it has
small Hom-sets. Moreover, $\cT/\cL$ has small coproducts and the
quotient functor $\fQ$ commutes with them by \cite[Corollary
  3.2.11]{N2}. Then it follows from Theorem 5.1.1 and
  Proposition 2.3.1 in  \cite{K2} that the functor $\fQ$ has a fully faithful
right adjoint $\fQ^R$ (hence $\fQ\comp\fQ^R\iso\id$).
\end{remark}

Although in general $\cT/\cL$ is not well generated by the set $\fQ(\cS)$ because (G2) does not hold, we have the following result.

\begin{prop}\label{inducedgen}
For $\cT$, $\cS$ and $\cL$ as above, the set $\fQ(\cS)$ satisfies {\rm (G1)} and {\rm (G3)} in $\cT/\cL$.
\end{prop}

\begin{proof}
Let $X\in\cT/\cL$ be such that $\cT/\cL(\fQ(S),X[j])=0$ for all $S\in\cS$ and all $j\in\ZZ$. Denoting by $\fQ^R$ the right adjoint of $\fQ$ (see \autoref{rmk:exadj}), we have
\[
\cT/\cL(\fQ(S),X[j])\iso\cT(S,\fQ^R(X)[j]).
\]
As $\cS$ satisfies (G1), it follows that $\fQ^R(X)\iso0$. Since $\fQ^R$ is fully faithful, this implies that $X\iso0$. Therefore $\fQ(\cS)$ satisfies (G1).

Observe moreover that, by \autoref{rmk:genalpha} (ii), the (small) set $\fQ(\cS)$ is contained in $(\cT/\cL)^\alpha$, for some regular cardinal $\alpha$. Hence it satisfies (G3), by \autoref{rmk:genalpha} (i). 
\end{proof}

\section{Dg categories and dg enhancements}\label{sect:dgenhancements}

In this section, we recall some general facts about dg categories and stick to the description of dg enhancements for well generated triangulated categories.

\subsection{A quick tour about dg categories}\label{subsect:dgdefinition}

An excellent survey about dg categories is \cite{K}. Nevertheless, we briefly summarize here what we need in the rest of the paper.

First of all, recall that a \emph{dg category} is a $\K$-linear category $\cC$
such that the morphism spaces
$\cC\left(A,B\right)$ are $\ZZ$-graded $\K$-modules with a differential
$d\colon\cC(A,B)\to\cC(A,B)$ of degree $1$ and the composition maps $\cC(B,C)\otimes_{\K}\cC(A,B)\to\cC(A,C)$
are morphisms of complexes, for all $A,B,C$ in $\cC$. By definition, the identity of each object is
a closed morphism of degree $0$.

\begin{ex}\label{ex:dgcat1}
(i) Any $\K$-linear category has a (trivial) structure of dg category,
with morphism spaces concentrated in degree $0$.

(ii) For a dg category $\cC$, one defines the opposite dg category $\cC\opp$ with the same objects as $\cC$ while $\cC\opp(A,B):=\cC\left(B,A\right)$. One should notice that, given two homogeneous elements $f\in\cC\opp(A,B)$ and $g\in\cC\opp(B,C)$, the composition $g\comp f$ in $\cC\opp$ is defined as the composition $(-1)^{\deg(f)\deg(g)}f\comp g$ in $\cC$.

(iii) Following \cite{Dr}, given a dg category $\cC$ and a full dg
subcategory $\cB$ of $\cC$, one can take the quotient
$\cC/\cB$ which is again a dg category.
\end{ex}

Given a dg category $\cC$ we denote by $\Ho(\cC)$ its
\emph{homotopy category}. To be precise, the objects of $\Ho(\cC)$ are the same
as those of $\cC$ while the morphisms from $A$ to $B$ are obtained
by taking the $0$-th cohomology $H^0(\cC\left(A,B\right))$ of the
complex $\cC\left(A,B\right)$.

\smallskip

A \emph{dg functor} $\fF\colon\cC_1\to\cC_2$ between two dg
categories is the datum of a map $\Ob(\cC_1)\to\Ob(\cC_2)$ and of
morphisms of complexes of $\K$-modules
$\cC_1\left(A,B\right)\to\cC_2\left(\fF(A),\fF(B)\right)$, for
$A,B\in\cC_1$, which are compatible with the compositions and
the units. Clearly, a dg functor $\fF\colon\cC_1\to\cC_2$ induces a functor
$\Ho(\fF)\colon\Ho(\cC_1)\to\Ho(\cC_2)$.

A dg functor $\fF\colon\cC_1\to\cC_2$ is a
\emph{quasi-equivalence}, if the maps
$\cC_1\left(A,B\right)\to\cC_2\left(\fF(A),\fF(B)\right)$ are quasi-isomorphisms, for
every $A,B\in\cC_1$, and $\Ho(\fF)$ is an equivalence.

One can consider the localization $\Hqe$ of the category of (small) dg
categories with respect to quasi-equivalences. Given a dg functor
$\fun{F}$, we will denote with the same symbol its image in
$\Hqe$. 
A morphism in $\Hqe$ is called a \emph{quasi-functor}. By the general theory of localizations and model categories (see, for example, \cite{K,To}), a quasi-functor between two dg categories $\cC_1$ and $\cC_2$ can be represented by a roof
\[
\xymatrix{
&\cC\ar[ld]_-{\fI}\ar[dr]^-{\fF}&\\
\cC_1&&\cC_2,
}
\]
where $\cC$ is a (cofibrant) dg category, $\fI$ is a quasi-equivalence
and $\fF$ is a dg functor. A quasi-functor $\fF$ in $\Hqe$ between the
dg categories $\cC_1$ and $\cC_2$ induces a functor
$\Ho(\fF)\colon\Ho(\cC_1)\to\Ho(\cC_2)$, well defined up to isomorphism.

Given a small dg category $\cC$, one can consider the dg category $\dgMod{\cC}$ of \emph{right dg
$\cC$-modules}. A right dg $\cC$-module is a dg functor
$\fM\colon\cC\opp\to\dgMod{\K}$, where $\dgMod{\K}$ is the dg
category of dg $\K$-modules. It is known that $\Ho(\dgMod{\cC})$ is, in a natural way, a triangulated category with small coproducts (see, for example, \cite{K}).

The full dg subcategory of acyclic right
dg modules is denoted by $\Ac(\cC)$, and $\Ho(\Ac(\cC))$ is a
localizing subcategory of the homotopy category
$\Ho(\dgMod{\cC})$. The objects of $\Ac(\cC)$ are the dg $\cC$-modules $\fM$ such that the complex $\fM(C)$ of $\K$-modules is acyclic, for all $C$ in $\cC$. A right dg $\cC$-module is \emph{representable} if it is contained
in the image of the Yoneda dg functor
\[
\dgYon[\cC]\colon\cC\to\dgMod{\cC}\qquad
A\mapsto\cC\left(\farg,A\right).
\]

The \emph{derived category} of the dg
category $\cC$ is the Verdier quotient
\[
\dgD(\cC):=\Ho(\dgMod{\cC})/\Ho(\Ac(\cC)),
\]
which turns out to be a triangulated category with small coproducts.
Following \cite{Dr}, one could first take the quotient $\dgMod{\cC}/\Ac(\cC)$ of the corresponding dg categories. Again by \cite{Dr}, there is a natural exact equivalence
\begin{equation}\label{Drinfeld}
\dgD(\cC)=\Ho(\dgMod{\cC})/\Ho(\Ac(\cC))\iso\Ho(\dgMod{\cC}/\Ac(\cC)).
\end{equation}

A right dg $\cC$-module is \emph{free} if it is isomorphic to a
small coproduct of dg modules of the form
$\dgYon[\cC](A)[m]$, where $A\in\cC$ and $m\in\ZZ$. A right dg $\cC$-module $\fM$ is \emph{semi-free} if it has a filtration
\begin{equation}\label{eqn:filt}
0=\fM_0\subseteq\fM_1\subseteq\ldots=\fM
\end{equation}
such that $\fM_j$ is a dg $\cC$-module,  $\fM_j/\fM_{j-1}$ is free, for all $j>0$, and $\fM$ is the colimit of the $\fM_j$'s. We denote by
$\SF{\cC}$ the full dg subcategory of $\dgMod{\cC}$ consisting of semi-free dg modules. Obviously the image of $\dgYon[\cC]\colon\cC\to\dgMod{\cC}$ is contained in $\SF{\cC}$.

\begin{remark}\label{rmk:pretr}
(i) It is easy to see that, for a dg category $\cC$, the homotopy category $\Ho(\SF{\cC})$ is a full triangulated subcategory of $\Ho(\dgMod{\cC})$. The dg category $\cC$ is called \emph{pretriangulated} if the essential image of the functor $\Ho(\dgYon[\cC])\colon\Ho(\cC)\to\Ho(\SF{\cC})$ is a triangulated subcategory.

(ii)  Given a quasi-functor $\fF\colon\cC_1\to\cC_2$ between two pretriangulated dg categories, the induced functor $\Ho(\fF)\colon\Ho(\cC_1)\to\Ho(\cC_2)$ is an exact functor between triangulated categories.

(iii) By \cite[Lemma B.3]{Dr}, there is a natural equivalence of triangulated categories $\Ho(\SF{\cC})\iso\dgD(\cC)$. We can actually be more precise about it. Indeed, the composition of natural dg functors
\[
\fH\colon\SF{\cC}\hookrightarrow\dgMod{\cC}\to\dgMod{\cC}/\Ac(\cC)
\]
is a quasi-equivalence. So, up to composing with \eqref{Drinfeld}, $\Ho(\fH)$ provides the exact equivalence $\Ho(\SF{\cC})\iso\dgD(\cC)$ mentioned above.
\end{remark}

If we are given a dg functor  $\fF\colon\cC_1\to\cC_2$, there exist dg functors
\[
\Ind(\fF)\colon\dgMod{\cC_1}\to\dgMod{\cC_2}
\qquad\Res(\fF)\colon\dgMod{\cC_2}\to\dgMod{\cC_1}.
\]
While $\Res(\fF)$ is simply defined by $\fM\mapsto\fM\comp\fF\opp$, the reader can have a look at \cite[Sect.\ 14]{Dr} for the explicit
definition and properties of $\Ind(\fF)$. Let us just observe that $\Ind(\fF)$ is left adjoint to $\Res(\fF)$, it preserves semi-free dg modules and
$\Ind(\fF)\colon\SF{\cC_1}\to\SF{\cC_2}$ is a quasi-equivalence if
$\fF\colon\cC_1\to\cC_2$ is such. Moreover, $\Ind(\fF)$ commutes with
the Yoneda embeddings, up to dg isomorphism.

\begin{ex}\label{ResSF}
Let $\cC$ be a dg category and $\cB$ a full dg subcategory of
$\cC$. Denoting by $\fI\colon\cB\mono\cC$ the inclusion dg functor,
the composition of dg functors
\[
\cC\mor{\dgYon[\cC]}\dgMod{\cC}\mor{\Res(\fI)}\dgMod{\cB}\to\dgMod{\cB}/\Ac(\cB)
\]
yields, in view of \autoref{rmk:pretr} (iii), a natural quasi-functor $\cC\to\SF{\cB}$.
\end{ex}

Let us give now the key definition for this paper.

\begin{definition}\label{def:enhancement}
A \emph{dg enhancement} (or simply an \emph{enhancement}) of a triangulated category $\cT$ is a pair
$(\cC,\fE)$, where $\cC$ is a pretriangulated
dg category and $\fE\colon\Ho(\cC)\to\cT$ is an exact
equivalence.
\end{definition}

A priori, one may have `different' enhancements for the same triangulated category. To make this precise, we need the following.

\begin{definition}\label{def:uniqueenh}
A triangulated category $\cT$
\emph{has a unique enhancement} if, given two enhancements $(\cC,\fE)$ and $(\cC',\fE')$ of $\cT$,
there exists a quasi-functor $\fF\colon\cC\to\cC'$ such that
$\Ho(\fF)$ is an exact equivalence.
\end{definition}

A concise way to say that a triangulated category $\cT$ has a unique enhancement is to say that, for any two  enhancements $(\cC,\fE)$ and $(\cC',\fE')$ of $\cT$, the dg categories $\cC$ and $\cC'$ are isomorphic in $\Hqe$. It is clear that the notion of uniqueness of dg enhancements forgets about part of the data in the definition of enhancement. In particular, the equivalence $\fE$ does not play a role. So, by abuse of notation, we will often simply say that $\cC$ is an enhancement of $\cT$ if there exists an exact equivalence $\Ho(\cC)\iso\cT$.

Nevertheless, there are stronger versions of the notion of uniqueness of dg enhancements. Indeed, we say that $\cT$ \emph{has a strongly unique} (respectively, \emph{semi-strongly unique}) \emph{enhancement} if moreover $\fF$ can be chosen so that there is an isomorphism of exact functors $\fE\iso\fE'\comp\Ho(\fF)$ (respectively, there is an isomorphism $\fE(C)\iso\fE'(\Ho(\fF)(C))$ in $\cT$, for every $C\in\cC$).

\begin{ex}\label{ex:derenh}
(i) If $\cC$ is a dg category, $\SF{\cC}$
is an enhancement of $\dgD(\cC)$.

(ii) Let $\cC$ be a pretriangulated dg category and let $\cB$ be a full pretriangulated dg subcategory of $\cC$. We mentioned already that, by the main result of \cite{Dr}, we have a natural exact equivalence between the Verdier quotient $\Ho(\cC)/\Ho(\cB)$ and $\Ho(\cC/\cB)$. Hence $\cC/\cB$, with the above equivalence, is an enhancement of $\Ho(\cC)/\Ho(\cB)$.
\end{ex}

\subsection{Dg enhancements for well generated triangulated categories}\label{subsect:dgwell}

If $\cC$ is a small dg category such that $\Ho(\cC)$ has
$\alpha$-coproducts, we denote by $\dgD_\alpha(\cC)$ the
\emph{$\alpha$-continuous derived category} of $\cC$, which is
defined as the full subcategory of $\dgD(\cC)$ with objects those $M\in\dgMod{\cC}$ such that the natural map
\[
(H^*(M))\left(\Plus_{i\in I}C_i\right)\longrightarrow\prod_{i\in I}(H^*(M))(C_i)
\]
(where the coproduct is intended in $\Ho(\cC)$) is an isomorphism, for all objects $C_i\in\cC$, with $\card{I}<\alpha$.
It is useful to know that $\dgD_\alpha(\cC)$ is also equivalent
to a quotient of $\dgD(\cC)$. More precisely, there is a localizing subcategory $\cN$ of
$\dgD(\cC)$ such that the quotient functor $\dgD(\cC)\to\dgD(\cC)/\cN$
restricts to an exact equivalence $\dgD_{\alpha}(\cC)\to\dgD(\cC)/\cN$ (see
\cite[Sect.\ 6]{P} for details). By \cite[Theorem 6.4]{P}, the triangulated category $\dgD_{\alpha}(\cC)$ is $\alpha$-compactly generated.

\begin{remark}\label{rmk:alphacont}
The triangulated category $\dgD_\alpha(\cC)$ has an obvious enhancement $\SFa{\cC}$ given as the full dg subcategory of $\SF{\cC}$ whose objects correspond to those in  $\dgD_\alpha(\cC)$, under the equivalence $\Ho(\SF{\cC})\iso\dgD(\cC)$ (see \autoref{rmk:pretr} (iii)). On the other hand, in a similar way, there is a an enhancement $\cN'$ of $\cN$ and, by \autoref{ex:derenh} (ii), the composition of dg functors
\[
\SFa{\cC}\mono\SF{\cC}\to\SF{\cC}/\cN'
\]
is a quasi-equivalence inducing the exact equivalence $\dgD_{\alpha}(\cC)\to\dgD(\cC)/\cN$. It follows that there is a natural quasi-functor $\SF{\cC}\to\SFa{\cC}$.
\end{remark}

The essential step in the proof of \cite[Theorem 7.2]{P} can be
reformulated (with a slight variant) as follows.

\begin{thm}\label{Porta}
Let $\cC$ be a pretriangulated dg category such that $\Ho(\cC)$ is well generated and let $\cB_0$ be a small set of objects in $\cC$. Then there exist a regular cardinal $\alpha$ and a small and full dg subcategory $\cB$ of $\cC$ containing $\cB_0$ such that $\Ho(\cB)$ is closed under $\alpha$-coproducts and the natural quasi-functor $\cC\to\SF{\cB}$ (see \autoref{ResSF}) induces an exact equivalence $\fPo\colon\Ho(\cC)\to\dgD_\alpha(\cB)$.
\end{thm}

\begin{proof}
It is shown in the proof of \cite[Theorem 7.2]{P} that all the required properties are satisfied, except possibly $\cB_0\subseteq\cB$, taking $\alpha$ such that $\Ho(\cC)$ is $\alpha$-compactly generated and $\cB$ such that $\Ho(\cB)\subseteq\Ho(\cC)^{\alpha}$ is a small subcategory and this inclusion is an equivalence. The conclusion then follows from \autoref{rmk:genalpha} (ii).
\end{proof}

\section{Some abstract results about exact functors}\label{sect:abstract}

In this section, we go back to the triangulated setting and prove some abstract results about exact functors which will be crucial in the rest of the paper.  This should be thought of as a rather technical but essential interlude towards the proof of \autoref{thm:crit1}. 

\smallskip

Let $\cA$ be a small category which we see here as a dg category sitting all in degree $0$ (see \autoref{ex:dgcat1} (i)). With a slight abuse of notation, we will identify $\dgD(\cA)$ with the homotopy category $\Ho(\SF{\cA})$ (see \autoref{rmk:pretr} (iii)). 

The fact that $\cA$ is in degree $0$ implies that an object of $\dgMod{\cA}$ can be regarded as a complex
\[
C=\{\cdots\to C^{j-1}\mor{d^{j-1}}C^j\mor{d^j}C^{j+1}\to\cdots\}
\]
in the abelian category $\Mod{\cA}$ of ($\K$-linear) functors $\cA\opp\to\Mod{\K}$ (where $\Mod{\K}$ is the abelian category of $\K$-modules).  Moreover, there is a natural exact equivalence $\dgD(\cA)\iso\D(\Mod{\cA})$ (see, for example, the beginning of \cite[Section 7]{LO} for a brief discussion about these issues). Under this natural identification, $\Ho(\dgYon)\colon\cA=\Ho(\cA)\to\dgD(\cA)$ is actually the composition of the usual Yoneda functor $\Yon\colon\cA\to\Mod{\cA}$ with the natural inclusion $\Mod{\cA}\hookrightarrow\D(\Mod{\cA})\iso\dgD(\cA)$. Therefore, for sake of simplicity, we denote by
\[
\Yon\colon\cA\to\dgD(\cA)
\]
the functor $\Ho(\dgYon)$.

As a consequence of the discussion above, it makes sense to say whether $C\in\dgMod{\cA}$ is bounded (above or below), and for every integer $n$ we can define the \emph{stupid truncations}
$\sigma_{\leq n}(C)$, $\sigma_{\geq n}(C)$ as
\begin{gather*}
\sigma_{\leq n}(C):=\{\cdots\to
C^j\mor{d^j}C^{j+1}\to\cdots\to C^{n-1}\mor{d^{n-1}}C^n \to0\to\cdots\}\\
\sigma_{\geq n}(C):=\{\cdots\to 0\to C^n\mor{d^{n}}C^{n+1}\to\cdots\to C^j\mor{d^j}C^{j+1}\to\cdots\}.
\end{gather*}
It is easy to see that, if $C\in\SF{\cA}$, then each $C^j$ is a free $\cA$-module and $\sigma_{\leq n}(C),\sigma_{\geq n}(C)\in\SF{\cA}$.

\begin{definition}\label{abc}
Let $\cT$ be a triangulated category with small coproducts. An exact functor $\fF\colon\dgD(\cA)\to\cT$ is \emph{right vanishing} if it preserves small coproducts and there exists a full subcategory $\cR$ of $\cT$ with the following properties:
\begin{enumerate}
\item[(R1)]\label{Rcoprod} $\cR$ is closed under small coproducts;
\item[(R2)]\label{Rext} $\cR$ is closed under extensions (meaning that, if $X\to Y\to Z$ is a distinguished triangle in $\cT$ with $X,Z\in\cR$, then $Y\in\cR$, as well); 
\item[(R3)]\label{Rrep} $\fF(\Yon(A))[k]\in\cR$ for every $A\in\cA$ and every integer $k<0$;
\item[(R4)]\label{Rort} $\cT\left(\fF(\Yon(A)),R\right)=0$ for every $A\in\cA$ and every $R\in\cR$.
\end{enumerate}
\end{definition}

\begin{ex}\label{ex:rightvancompact}
Let $\cT$ be a triangulated category with small coproducts and let $\fF\colon\dgD(\cA)\to\cT$ be an exact functor that preserves small coproducts. Assume further that $\fF(\Yon(A))\in\cT^c$, for every $A\in\cA$, and that
\[
\cT\left(\fF(\Yon(A)),\fF(\Yon(A'))[k]\right)=0,
\]
for every $A,A'\in\cA$ and every $k<0$. Then $\fF$ is right vanishing. Indeed, we can take $\cR$ to be the full subcategory of $\cT$  consisting of those $R$ such that $\cT\left(\fF(\Yon(A)),R\right)=0$ for every $A\in\cA$.
\end{ex}

Now we can prove the following results, which should be compared to Lemma 3.2, Corollary 3.3 and Proposition 3.4 in \cite{LO}.

\begin{lem}\label{prop:f1}
Let $\cT$ be a triangulated category with small coproducts and let $\fF\colon\dgD(\cA)\to\cT$ be a right vanishing functor. Then, given $A\in\cA$ and $C\in\dgD(\cA)$,
\[
\cT\left(\fF(\Yon(A)),\fF(\sigma_{\geq n}(C))[k]\right)=0
\]
for all integers $k<n-1$, and also for $k=n-1$ if $C$ is bounded above.
\end{lem}

\begin{proof}
Let $\cR$ be a full subcategory of $\cT$ as in \autoref{abc}. If $C\in\SF{\cA}$ has filtration $\{C_j\}$, the induced filtration $\{C_j'\}$ of $\sigma_{\geq n}(C)$ has clearly the property that each quotient $C'_{j}/C'_{j-1}$ is isomorphic to a small coproduct of objects of the form $\Yon(A)[s]$, for $A\in\cA$ and $s\leq-n$. As $\fF$ preserves small coproducts and $\cR$ satisfies (R1) and (R3), it follows immediately that $\fF(C'_j/C'_{j-1})[k]\in\cR$ for $k<n$. Since $\cR$ satisfies (R2), from the distinguished triangle $C'_{j-1}\to C'_j\to C'_j/C'_{j-1}$ of $\dgD(\cA)$ we deduce by induction on $j$ that
\begin{equation}\label{filtvan}
\fF(C'_j)[k]\in\cR\qquad\text{for $j\ge0$ and $k<n$.}
\end{equation}
Now we use the fact that $\sigma_{\geq n}(C)\iso\hocolim(C'_j)$. Recall that, if we denote by $s_j\colon C'_j\to C'_{j+1}$ the inclusion morphisms, then $\hocolim(C'_j)$ is, by definition, the cone in $\dgD(\cA)$ of the morphism
\[
\sum_{j\geq 0}(\id_{C'_j}-s_j)\colon\Plus_{j\geq 0}C_j'\longrightarrow\Plus_{j\geq 0}C_j'.
\]
As $\fF$ preserves small coproducts, we have an isomorphism in $\cT$
\[
\fF(\sigma_{\geq n}(C))\iso\fF(\hocolim(C'_j))\iso\hocolim\fF(C'_j),
\]
hence a distinguished triangle
\[
\Plus_{j\ge0}\fF(C'_j)\longrightarrow\Plus_{j\ge0}\fF(C'_j)\longrightarrow\fF(\sigma_{\geq n}(C)).
\]
Using \eqref{filtvan} and, again, the fact that $\cR$ satisfies (R1) and (R2), we obtain that $\fF(\sigma_{\geq n}(C))[k]\in\cR$ for $k<n-1$.

If $C$ is bounded above, we can take $C_j=\sigma_{\ge t-j}(C)$ for some integer $t$. Then $\sigma_{\geq n}(C)=C_{t-n}=C'_{t-n}$ (meaning $0$ if $t-n<0$), whence
also $\fF(\sigma_{\geq n}(C))[n-1]\in\cR$ by \eqref{filtvan}.

In both cases, the thesis then follows immediately from the fact that $\cR$ satisfies (R4).
\end{proof}

\begin{cor}\label{cor:f2}
Let $\cT$ be a triangulated category with small coproducts and let $\fF\colon\dgD(\cA)\to\cT$ be a right vanishing functor. Then, given $A\in\cA$ and $C\in\dgD(\cA)$, the map (induced by the natural morphism $C\to\sigma_{\leq m}(C)$)
\[
\cT(\fF(\Yon(A)),\fF(C))\longrightarrow\cT(\fF(\Yon(A)),\fF(\sigma_{\leq m}(C)))
\]
is injective for every integer $m>0$, and also for $m=0$ if $C$ is bounded above. Moreover, the map is an isomorphism for $m>1$, and also for $m=1$ if $C$ is bounded above.
\end{cor}

\begin{proof}
For every integer $m$ we have a distinguished triangle
\[
\sigma_{\geq m+1}(C)\longrightarrow C\longrightarrow\sigma_{\leq m}(C)
\]
in $\dgD(\cA)$. It is then enough to apply the cohomological functor $\cT(\fF(\Yon(A)),\fF(\farg))$ to it, taking into account that, by \autoref{prop:f1},
\[
\cT(\fF(\Yon(A)),\fF(\sigma_{\geq m+1}(C)))=0
\]
for $m>0$ (also for $m=0$ if $C$ is bounded above) and
\[
\cT(\fF(\Yon(A)),\fF(\sigma_{\geq m+1}(C))[1])=0
\]
for $m>1$ (also for $m=1$ if $C$ is bounded above).
\end{proof}

\begin{prop}\label{prop:f3}
Let $\cT$ be a triangulated category with small coproducts and let $\fF_1,\fF_2\colon\dgD(\cA)\to\cT$ be right vanishing functors. Assume moreover that there is an isomorphism of functors $\theta\colon\fF_1\comp\Yon\to\fF_2\comp\Yon$. Then, for every $C\in\dgD(\cA)$ bounded above, there exists an isomorphism $\theta'_C:\fF_1(C)\to\fF_2(C)$ such that, for every $A\in\cA$, every $k\in\ZZ$ and every $f\in\dgD(\cA)(\Yon(A)[k],C)$, the diagram
\begin{equation}\label{eqn:quad}
\xymatrix{
\fF_1(\Yon(A)[k])\ar[rr]^{\fF_1(f)}\ar[d]_{\theta_A[k]}&& \fF_1(C)\ar[d]^{\theta'_C}\\
\fF_2(\Yon(A)[k])\ar[rr]^{\fF_2(f)}&& \fF_2(C)
}
\end{equation}
commutes in $\cT$.
\end{prop}

\begin{proof}
The proof proceeds verbatim as the proof of \cite[Proposition 3.4]{LO}. It should be noted that in \cite{LO} the authors assume further that $\fF_i(\Yon(A))$ is a compact object, for all $A\in\cA$ and $i=1,2$. They need this hypothesis in the proof of \cite[Proposition 3.4]{LO} only to invoke \cite[Corollary 3.3]{LO}. But, under our assumptions, \cite[Corollary 3.3]{LO} can be replaced by \autoref{cor:f2}, for $C\in\SF{\cA}$ bounded above.
\end{proof}

\section{Uniqueness of enhancements: a general criterion}\label{sect:firstcrit}

This section is completely devoted to the proof of \autoref{thm:crit1}. Hence, let $\cA$ be a small category which we see here as a dg category sitting all in degree $0$ and let $\cL$ be a localizing subcategory of $\dgD(\cA)$. We will always assume that
\begin{itemize}
\item[(a)] The quotient $\dgD(\cA)/\cL$ is a well generated triangulated category;
\item[(b)] The quotient functor $\fQ\colon\dgD(\cA)\to\dgD(\cA)/\cL$ is right vanishing;
\end{itemize}
as in the hypotheses of \autoref{thm:crit1}.

\subsection{The quasi-functor}\label{subsect:quasifun1}

Assume that there exists an exact equivalence $\fE\colon\dgD(\cA)/\cL\to\Ho(\cC)$, for some pretriangulated dg category $\cC$. Notice that (a) clearly implies that $\Ho(\cC)$ is also a well generated triangulated category. Consider the composition of functors
\[
\xymatrix{
\fH\colon\cA\ar[r]^-{\Yon}&\dgD(\cA)\ar[r]^-{\fQ}&\dgD(\cA)/\cL\ar[r]^-{\fE}&\Ho(\cC).
}
\]

\begin{remark}\label{gen}
As $\Yon(\cA)$ is a small set of compact generators of $\dgD(\cA)$ (see \cite[Example 1.9]{LO}), it follows from \autoref{inducedgen} that $\fQ\comp\Yon(\cA)$ satisfies (G1) in $\dgD(\cA)/\cL$. Since $\fE$ is an exact equivalence, also $\fH(\cA)$ satisfies (G1) in $\Ho(\cC)$.
\end{remark}

Denoting by $\cB_0$ the full dg subcategory of $\cC$ such that $\Ho(\cB_0)=\fH(\cA)$, we can clearly regard $\fH$ as a functor $\cA\to\Ho(\cB_0)$.

Let $\tau_{\leq 0}(\cB_0)$ be the dg category with the same objects as $\cB_0$ and with
\[
\tau_{\le0}(\cB_0)\left(B_1,B_2\right):=\tau_{\le0}\left(\cB_0\left(B_1,B_2\right)\right)
\]
for every $B_1$ and $B_2$ in $\cB_0$. Here, for a complex of $\K$-modules (or, more generally, of objects in an abelian category)
\[
C=\{\cdots\to C^{j-1}\mor{d^{j-1}}C^j\mor{d^j}C^{j+1}\to\cdots\},
\]
and for every integer $n$, we define
\begin{gather*}
\tau_{\le n}(C):=\{\cdots\to
C^j\mor{d^j}C^{j+1}\to\cdots\to C^{n-1}\to\ker d^n\to0\to\cdots\},\\
\tau_{\ge n}(C):=\{\cdots\to0\to\cok d^{n-1}\to C^{n+1}\to\cdots\to
C^j\mor{d^j}C^{j+1}\to\cdots\}.
\end{gather*}
There are obvious dg functors $\tau_{\le0}(\cB_0)\to\Ho(\cB_0)$ and $\tau_{\le0}(\cB_0)\to\cB_0$, and the former is a quasi-equivalence thanks to (b) (taking into account that $\fE$ is an exact equivalence). Thus we obtain a quasi-functor $\Ho(\cB_0)\to\cB_0$.

By \autoref{Porta} there exist a regular cardinal $\alpha$ and a small and full dg subcategory $\cB$ of $\cC$ containing $\cB_0$ such that $\Ho(\cB)$ is closed under $\alpha$-coproducts and the natural quasi-functor $\cC\to\SF{\cB}$ induces an exact equivalence $\fPo\colon\Ho(\cC)\to\dgD_\alpha(\cB)$.

If we compose $\fH\colon\cA\to\Ho(\cB_0)$ with the quasi-functor $\Ho(\cB_0)\to\cB_0$ and the natural inclusion $\cB_0\mono\cB$, we get a quasi-functor $\fH'\colon\cA\to\cB$. From it we finally obtain a quasi-functor
\[
\xymatrix{
\fG_1\colon\SF{\cA}\ar[rr]^-{\Ind(\fH')}&&\SF{\cB}\ar[rr]^-{\fQ'}&&\SFa{\cB},
}
\]
where $\fQ'$ denotes the natural quasi-functor described in \autoref{rmk:alphacont}. By passing to the homotopy categories, we have also the exact functor
\[
\fF_1:=\Ho(\fG_1)\colon\dgD(\cA)\longrightarrow\dgD_\alpha(\cB).
\]
On the other hand, we can proceed differently and take the exact functor
\[
\xymatrix{
\fF_2\colon\dgD(\cA)\ar[r]^-{\fQ}&\dgD(\cA)/\cL\ar[r]^-{\fE}&\Ho(\cC)\ar[r]^-{\fPo}&\dgD_\alpha(\cB).
}
\]

The following results will be used later.

\begin{lem}\label{lem:C12}
The functors $\fF_1$ and $\fF_2$ satisfy the assumptions of \autoref{prop:f3}.
\end{lem}

\begin{proof}
It is obvious that $\fF_2$ commutes with small coproducts if and only if $\fQ$ does. But this last fact was already observed in \autoref{rmk:exadj}. On the other hand, the exact functor $\Ho(\Ind(\fH'))$ preserves small coproducts because $\Ind(\fH')$ is a left adjoint. By \autoref{rmk:alphacont} and \autoref{rmk:exadj}, the exact functor $\Ho(\fQ')$ (and thus $\fF_1$) preserves small coproducts as well. Moreover, $\fF_2$ is right vanishing because $\fQ$ is right vanishing by (b). Finally, it is clear by construction that $\fF_1\comp\Yon\iso\fF_2\comp\Yon$, and this obviously implies that also $\fF_1$ is right vanishing.
\end{proof}

\begin{cor}\label{F1G1}
The set $\fF_1\comp\Yon(\cA)$ satisfies {\rm (G1)} in $\dgD_\alpha(\cB)$.
\end{cor}

\begin{proof}
As $\fF_1\comp\Yon\iso\fF_2\comp\Yon$, it is enough to show that $\fF_2\comp\Yon(\cA)$ satisfies (G1) in $\dgD_\alpha(\cB)$. Since $\fF_2\comp\Yon=\fPo\comp\fH$ and $\fPo$ is an exact equivalence, this follows from \autoref{gen}.
\end{proof}

\subsection{The proof of \autoref{thm:crit1}}\label{subsect:proofcrit1}

Let $\cC$ and $\fE\colon\dgD(\cA)/\cL\to\Ho(\cC)$ be as in \autoref{subsect:quasifun1}.
Denote by $\cL'$ the full dg subcategory of $\SF{\cA}$ such that $\Ho(\cL')\iso\cL$ under the equivalence $\Ho(\SF{\cA})\iso\dgD(\cA)$.

\begin{lem}\label{lem:fact1}
The quasi-functor $\fG_1$ factors through the quotient dg functor $\SF{\cA}\to\SF{\cA}/\cL'$.
\end{lem}

\begin{proof}
The proof is very similar to the one of \cite[Lemma 5.2]{LO} with the required adjustments due to the more general setting we are working in. More precisely, in view of the main result of \cite{Dr}, it is enough to show that $\fF_1$ factors through the quotient $\dgD(\cA)/\cL$, i.e.\ that $\fF_1(L)\iso 0$, for all $L$ in $\cL$.

By \autoref{F1G1}, we have just to show that
\[
\dgD_\alpha(\cB)(\fF_1(\Yon(A)),\fF_1(L))=0,
\]
for all $A\in\cA$ and $L\in\cL$ (since $\cL$ is closed under shifts).

By \autoref{lem:C12}, we can apply the results of \autoref{sect:abstract} to the functors $\fF_1$ and $\fF_2$. In particular, by \autoref{cor:f2}, there are isomorphisms (for $i=1,2$)
\[
\dgD_\alpha(\cB)(\fF_i(\Yon(A)),\fF_i(L))\iso\dgD_\alpha(\cB)(\fF_i(\Yon(A)),\fF_i(\sigma_{\leq m}(L)))
\]
for $m>1$. On the other hand, by \autoref{prop:f3}, there is an isomorphism $\fF_1(\sigma_{\leq m}(L))\iso\fF_2(\sigma_{\leq m}(L))$.
It follows that
\[
\dgD_\alpha(\cB)(\fF_1(\Yon(A)),\fF_1(L))\iso
\dgD_\alpha(\cB)(\fF_2(\Yon(A)),\fF_2(L)),
\]
and the latter Hom-space is naturally isomorphic to $0$, since $\fF_2(L)\iso0$.
\end{proof}

Hence $\fG_1$ factors through a quasi-functor $\fG\colon\SF{\cA}/\cL'\to\SFa{\cB}$. If we show that it defines an isomorphism in $\Hqe$, \autoref{thm:crit1} would follow immediately, taking into account \autoref{Porta}. This is the content of the next proposition.

\begin{prop}\label{prop:keycrit}
In the above situation, $\fG\colon\SF{\cA}/\cL'\to\SFa{\cB}$ defines an isomorphism in $\Hqe$.
\end{prop}

\begin{proof}
Setting $\fF:=\Ho(\fG)\colon\dgD(\cA)/\cL\to\dgD_\alpha(\cB)$, it is enough to show that $\fF$ is an equivalence. Notice that $\fF$ preserves small coproducts, since the same is true for $\fF\comp\fQ\iso\fF_1$ by \autoref{lem:C12}.

We first prove that $\fF$ is fully faithful, namely that the map
\[
\fF_{B,C}\colon\dgD(\cA)/\cL\left(B,C\right)\to\dgD_\alpha(\cB)\left(\fF(B),\fF(C)\right)
\]
is an isomorphism for all $B,C\in\dgD(\cA)/\cL$. Now, it is easy to see that the full subcategory $\cS$ of $\dgD(\cA)/\cL$ which consists of the objects $S$ for which $\fF_{S,C}$ is an isomorphism for every $C\in\dgD(\cA)/\cL$ is a localizing subcategory of $\dgD(\cA)/\cL$. In view of \autoref{gen} and \autoref{prop:PortaG1}, it is therefore enough to prove that $\fF_{\fQ(\Yon(A)),C}$ is an isomorphism, for all $A\in\cA$ and $C\in\dgD(\cA)/\cL$. The proof of this fact proceeds as in \cite[Lemma 5.3]{LO}. Let us outline the argument here for the convenience of the reader. 

Setting $P:=\fQ^R(C)$, we have $\fQ(P)\iso C$, and so $\fF_{\fQ(\Yon(A)),C}$ is an isomorphism if and only if $\fF_{\fQ(\Yon(A)),\fQ(P)}$ is.
Moreover, as the map
\[
\fQ_{\Yon(A),P}\colon\dgD(\cA)\left(\Yon(A),P\right)\to\dgD(\cA)/\cL\left(\fQ(\Yon(A)),\fQ(P)\right)
\]
is an isomorphism by adjunction (since $\fQ^R(\fQ(P))\iso P$) and $\fF_1\iso\fF\comp\fQ$, we can just prove that
\[
(\fF_1)_{\Yon(A),P}\colon\dgD(\cA)\left(\Yon(A),P\right)\to\dgD_\alpha(\cB)\left(\fF_1(\Yon(A)),\fF_1(P)\right)
\]
is an isomorphism.

By \autoref{lem:C12} the functors $\fF_1$ and $\fF_2$ satisfy the hypotheses of  \autoref{cor:f2}, and the same is true of $\fQ$ by (b). In particular, there is an isomorphism
\[
\dgD(\cA)/\cL\left(\fQ(\Yon(A)),\fQ(P)\right)\iso\dgD(\cA)/\cL\left(\fQ(\Yon(A)),\fQ(\sigma_{\leq m}(P))\right),
\]
for $m>1$. Moreover, this is compatible with the natural isomorphism
\[
\dgD(\cA)(\Yon(A),P)\iso\dgD(\cA)(\Yon(A),\sigma_{\leq m}(P)).
\]
As $\fQ_{\Yon(A),P}$ is an isomorphism, it follows that $\fQ_{\Yon(A),\sigma_{\leq m}(P)}$ is an isomorphism, too.

The same argument applies to the functor $\fF_1$, and then it is enough to check that $(\fF_1)_{\Yon(A),\sigma_{\leq m}(P)}$ is an isomorphism. To this purpose, consider the commutative diagram
\[
\xymatrix{
\dgD(\cA)(\Yon(A),\sigma_{\leq m}(P)) \ar[rrr]^-{(\fF_1)_{\Yon(A),\sigma_{\leq m}(P)}} \ar[drrr]_-{(\fF_2)_{\Yon(A),\sigma_{\leq m}(P)}} & & & \dgD_\alpha(\cB)\left(\fF_1(\Yon(A)),\fF_1(\sigma_{\leq m}(P))\right) \ar[d]^-{\gamma} \\
 & & & \dgD_\alpha(\cB)\left(\fF_2(\Yon(A)),\fF_2(\sigma_{\leq m}(P))\right),
}
\]
where the existence of an isomorphism $\gamma$ is ensured by \autoref{prop:f3} (which, again, applies due to \autoref{lem:C12}). Since $\fF_2=\fPo\comp\fE\comp\fQ$, the fact that $\fQ_{\Yon(A),\sigma_{\leq m}(P)}$ is an isomorphism implies that also $(\fF_2)_{\Yon(A),\sigma_{\leq m}(P)}$ is an isomorphism, taking into account that $\fPo\comp\fE$ is an equivalence. In  conclusion, $(\fF_1)_{\Yon(A),\sigma_{\leq m}(P)}$ is an isomorphism as well.

Finally, the essential image of $\fF$ is a localizing subcategory of $\dgD_\alpha(\cB)$ (because $\fF$ preserves small coproducts) which contains $\fPo\comp\fH(\cA)$ (since $\fPo\comp\fH=\fF_2\comp\Yon\iso\fF\comp\fQ\comp\Yon$). As $\fPo$ is an exact equivalence, it follows from \autoref{gen} and \autoref{prop:PortaG1} that $\fF$ is essentially surjective.
\end{proof}

\section{The case of the derived category of a Grothendieck category}\label{sect:Groth}

In this section, we prove \autoref{thm:main1} and discuss some geometric applications of this abstract criterion for Grothendieck categories.

\subsection{The abstract result}\label{subsect:Grothabstr}

Let $\cG$ be a Grothendieck category and let $\cA$ be a full
subcategory of $\cG$ whose objects form a small set of generators of
$\cG$. Setting $\cM:=\Mod{\cA}$, we will denote by
$\fS\colon\cG\to\cM$ the natural functor defined by
\[
\fS(C)(A):=\cG(A,C),
\]
for $C\in\cG$ and $A\in\cA$.

We can first prove the following result which should be compared to \cite[Theorem 7.4]{LO}.

\begin{prop}\label{prop:adj}
The functor $\fS\colon\cG\to\cM$ admits a left adjoint
$\fT\colon\cM\to\cG$. Moreover, $\fT$ is exact, $\fT\comp\fS\iso\id_{\cG}$, $\cN:=\ker\fT$ is a localizing Serre subcategory of
$\cM$ and $\fT$ induces an equivalence $\fT'\colon\cM/\cN\to\cG$ such
that $\fT\iso\fT'\comp\Pi$, where $\Pi\colon\cM\to\cM/\cN$ is the
projection functor.
\end{prop}

\begin{proof}
In \cite[Theorem 2.2]{CENT} the analogous statement is proved for the
functor $\fS'\colon\cG\to\MOD{R}$, which we are going to define.
Consider the object $U:=\Plus_{A\in\cA}A$ of $\cG$ and denote, for
every $A\in\cA$, by $\inc{A}\colon A\mono U$ and $\pro{A}\colon U\epi
A$ the natural inclusion and projection morphisms, respectively. Let
$S$ be the ring (with unit) $\cG(U,U)$ and $R$ the subring of
$S$ consisting of those $s\in S$ for which $s\comp\inc{A}\ne0$ only
for a finite number of $A\in\cA$. Notice that $R$ is a ring with unit
if and only if $\cA$ has a finite number of objects, in which case
obviously $R=S$. Let moreover $\MOD{R}$ be the full subcategory of
$\Mod{R}$ having as objects those $P\in\Mod{R}$ for which $PR=P$
(clearly $\MOD{R}=\Mod{R}=\Mod{S}$ if $\cA$ has a finite number of
objects). Then $\fS'$ is simply given as the composition of
$\cG(U,\farg)\colon\cG\to\Mod{S}$ with the natural functor
$\Mod{S}\to\MOD{R}$ defined on objects by $P\mapsto PR$. To deduce our
statement from \cite[Theorem 2.2]{CENT} it is therefore enough to show
that there is an equivalence of categories $\fE\colon\cM\to\MOD{R}$
such that $\fS'\iso\fE\comp\fS$.

In order to define $\fE$, consider first an object $M$ of $\cM$,
namely a ($\K$-linear) functor $M\colon\cA\opp\to\Mod{\K}$. As a
$\K$-module $\fE(M)$ is just $\Plus_{A\in\cA}M(A)$, whereas the
$R$-module structure is defined as follows. Given $r\in R$ and
$m\in\fE(M)$ with components $m_A\in M(A)$ for every $A\in\cA$, the
element $mr\in\fE(M)$ has components
$(mr)_A=\sum_{B\in\cA}M(\pro{B}\comp r\comp\inc{A})(m_B)$. It is easy
to prove that this actually defines an object $\fE(M)$ of
$\MOD{R}$. As for morphisms, given $M,M'\in\Mod{\cA}$ and a natural
transformation $\gamma\colon M\to M'$, the morphism of $R$-modules
$\fE(\gamma)\colon\fE(M)\to\fE(M')$ sends $m$ to $m'$, where
$m'_A:=\gamma(A)(m_A)$ for every $A\in\cA$. It is not difficult to
check that this really defines a functor $\fE\colon\cM\to\MOD{R}$ and
that $\fS'\iso\fE\comp\fS$.

It remains to prove that $\fE$ is an equivalence. It is clear by
definition that $\fE$ is faithful. As for fullness, given
$M,M'\in\Mod{\cA}$ and a morphism $\phi\colon\fE(M)\to\fE(M')$ in
$\MOD{R}$, it is easy to see that $\phi=\fE(\gamma)$, where
$\gamma\colon M\to M'$ is the natural transformation defined as
follows. For every $A\in\cA$ and for every $a\in M(A)$, denoting by
$m$ the element of $\fE(M)$ such that $m_A=a$ and $m_B=0$ for $A\ne
B\in\cA$, we set $\gamma(A)(a):=\phi(m)_A$. Finally, $\fE$ is
essentially surjective because it is not difficult to prove that for
every $P\in\MOD{R}$ we have $P\iso\fE(M)$ with $M\in\cM$ defined in
the following way. Setting $r_f:=\inc{B}\comp f\comp\pro{A}\in R$ for
every morphism $f\colon A\to B$ of $\cA$, we define $M(A):=Pr_{\id_A}$
for every $A\in\cA$, whereas $M(f)\colon M(B)=Pr_{\id_B}\to
M(A)=Pr_{\id_A}$ for every morphism $f\colon A\to B$ of $\cA$ is given
by $pr_{\id_B}\mapsto pr_f=(pr_f)r_{\id_A}$ for every $p\in P$.
\end{proof}

\begin{remark}\label{rmk:conseq}
It should be noted that while, by \autoref{prop:adj}, the functor $\fT$ is exact, $\fS$ is only left-exact in general. On the other hand, the fact that $\fT\comp\fS\iso\id_\cG$ implies that $\fS$ is fully faithful and that $\fT\comp\Yon$ is isomorphic to the inclusion $\cA\mono\cG$ (since $\fS\rest{\cA}\iso\Yon$ by definition). Here, as in the previous sections, $\Yon\colon\cA\mono\cM\mono\D(\cM)$ is the Yoneda embedding.
\end{remark}

Passing from $\cG$ to its derived category $\D(\cG)$, we observe that the functors $\fT$, $\fT'$ and $\Pi$ being exact, we can denote by the
same letters the corresponding derived functors.

Denote by $\D_{\cN}(\cM)$ the full triangulated subcategory of $\D(\cM)$ consisting of complexes with cohomology in $\cN$. Let moreover $\fQ\colon\D(\cM)\to\D(\cM)/\D_{\cN}(\cM)$ be the projection functor.

\begin{cor}\label{restfun}
The functor $\Pi$ induces an exact equivalence
$\Pi'\colon\D(\cM)/\D_{\cN}(\cM)\to\D(\cM/\cN)$ such that
$\Pi\iso\Pi'\comp\fQ$. Moreover, $\fT'\comp\Pi'\comp\fQ\comp\Yon$ is isomorphic to the inclusion $\cA\mono\cG\mono\D(\cG)$.
\end{cor}

\begin{proof}
By \autoref{prop:adj}, $\Pi\colon\cM\to\cM/\cN$ admits a right adjoint, so the first part of the statement follows from \cite[Lemma 5.9]{K0}. Hence the diagram
\[\xymatrix{ & & \D(\cM) \ar[dll]_-{\fQ} \ar[d]_{\Pi} \ar[drr]^-{\fT} \\
\D(\cM)/\D_{\cN}(\cM) \ar[rr]_-{\Pi'} & & \D(\cM/\cN) \ar[rr]_{\fT'}
& & \D(\cG)}\]
commutes up to isomorphism. By the
above commutativity, the second part of the statement follows from the fact that the inclusion
$\cA\mono\cG$ is isomorphic to $\fT\comp\Yon$ (see \autoref{rmk:conseq}).
\end{proof}

We are now ready to prove our first result.

\begin{proof}[Proof of \autoref{thm:main1}]
Given a Grothendieck category $\cG$, by \autoref{prop:adj} and \autoref{restfun} we know that $\D(\cG)\iso\D(\cM)/\D_{\cN}(\cM)$, for $\cM=\Mod{\cA}$ and $\cN$ defined as above.

Consider $\cA$, consisting of a small set of generators of $\cG$, as a small dg category all sitting in degree $0$. As noted at the beginning of \autoref{sect:abstract}, there is a natural exact equivalence $\D(\cM)\iso\dgD(\cA)$. By setting $\cL$ to be the full localizing subcategory of $\dgD(\cA)$ which is the image of $\D_\cN(\cM)$ under the above equivalence, we have that $\D(\cG)\iso\dgD(\cA)/\cL$.

Let us observe the following:
\begin{itemize}
\item[(a)] The quotient $\dgD(\cA)/\cL$ is a well generated triangulated category. This is because $\D(\cG)$, which is naturally equivalent to $\dgD(\cA)/\cL$, is well generated by \autoref{ex:alphacompgen} and well generation is obviously  preserved under exact equivalences.
\item[(b)] The quotient functor $\fQ\colon\dgD(\cA)\to\dgD(\cA)/\cL$ is right vanishing. Indeed, one can consider the full subcategory $\cR$ of $\dgD(\cA)/\cL$ whose objects, under the equivalences $\dgD(\cA)/\cL\iso\D(\cM)/\D_{\cN}(\cM)\iso\D(\cG)$ described above, correspond to objects in $\D(\cG)$ with cohomology in strictly positive degrees. It is very easy to check that $\cR$ satisfies the properties of \autoref{abc}, taking into account that, by \autoref{restfun}, the objects $\fQ(\Yon(A))$ (for $A$ in $\cA$) correspond to objects in the abelian category $\cG$.
\end{itemize}
In particular, the assumptions of \autoref{thm:crit1} are satisfied and we can apply that result concluding that the triangulated category $\dgD(\cA)/\cL$ (and hence $\D(\cG)$) has a unique enhancement.
\end{proof}

\subsection{The geometric examples}\label{subsect:Grothex}

We discuss now some geometric incarnations of \autoref{thm:main1}. There are certainly many interesting geometric triangulated categories which are equivalent to the derived category of a Grothendieck category and which are not considered here. So we do not claim that our list of applications is complete. Notice that, beyond the geometric situations studied in \cite{LO} and described in the introduction, the uniqueness of enhancements has been investigated in other cases, e.g.\ for the derived categories of supported quasi-coherent sheaves in special situations (see \cite[Lemma 4.6]{CS}).

\smallskip

\noindent{\bf Algebraic stacks.} Let $X$ be an algebraic stack. For general facts about these geometric objects, we refer to \cite{LMB} and \cite{SP}.

We can consider the abelian categories $\Mod{\ko_X}$ of $\ko_X$-modules on $X$ and $\Qcoh(X)$ of quasi-coherent $\ko_X$-modules on $X$. The fact that $\Qcoh(X)$ is a Grothendieck category is proved in \cite[Tag 06WU]{SP}. Passing to the derived categories, we can consider $\D(\Qcoh(X))$ and the full triangulated subcategory $\Dq(X)$ of $\D(\Mod{\ko_X})$ consisting of complexes with quasi-coherent cohomology. The relation between these two triangulated categories is delicate, as pointed out in \cite[Theorem 1.2]{HNR}.

We then have the following.

\begin{cor}\label{cor:geo1}
If $X$ is an algebraic stack, then $\D(\Qcoh(X))$ has a unique enhancement. If $X$ is also quasi-compact and with quasi-finite affine diagonal, then $\Dq(X)$ has a unique enhancement.
\end{cor}

\begin{proof}
The first part of the statement is an obvious consequence of \autoref{thm:main1}. For the second part, observe that, by \cite[Theorem A]{HR}, the category $\Dq(X)$ is compactly generated by a single object. Hence, by \cite[Theorem 1.2]{HNR}, the natural functor $\D(\Qcoh(X))\to\Dq(X)$ is an exact equivalence.
\end{proof}

\begin{remark}\label{rmk:schemes1}
The above result specializes to the case of schemes. In particular, $\D(\Qcoh(X))$ has a unique enhancement for any scheme $X$. If $X$ is quasi-compact and semi-separated (i.e.\ the diagonal is affine or, equivalently, the intersection of two open affine subschemes in $X$ is affine), then $\D(\Qcoh(X))\iso\Dq(X)$ (see \cite[Corollary 5.5]{BN}) and the same uniqueness result holds for $\Dq(X)$. This extends vastly the results in \cite{LO}, where the uniqueness results for both categories are proved only for quasi-compact, semi-separated schemes with enough locally free sheaves. This last condition means that for any finitely presented sheaf $F$ there is an epimorphism $E\epi F$ in $\Qcoh(X)$, where $E$ is locally free of finite type.

As in \cite[Remark 7.7]{LO}, we should observe here that we can take $X$ to be a semi-separated scheme rather than separated, because the proof of \cite[Corollary 5.5]{BN} works for a semi-separated scheme as well.
\end{remark}

\smallskip

\noindent{\bf Twisted sheaves.} Let $X$ be a scheme and pick
$\alpha\in H^2_{\text {\'et}}(X,\ko_X^*)$, i.e.\ an
element in the Brauer group $\Br(X)$ of $X$. We may represent $\alpha$ by a \v{C}ech
2-cocycle $\{\alpha_{ijk}\in\Gamma(U_i\cap U_j\cap
U_k,\ko^*_X)\}$ with $X=\bigcup_{i\in I} U_i$  an
appropriate  open cover in the \'etale topology. An \emph{$\alpha$-twisted
quasi-coherent sheaf} $E$ consists of pairs $(\{E_i\}_{i\in
I},\{\varphi_{ij}\}_{i,j\in I})$ such that the $E_i$ are
quasi-coherent sheaves on $U_i$ and $\varphi_{ij}:{E}_j|_{U_i\cap
U_j}\rightarrow{E}_i|_{U_i\cap U_j}$ are isomorphisms satisfying
the following conditions:
\begin{itemize}
\item $\varphi_{ii}=\id$;
\item $\varphi_{ji}=\varphi_{ij}^{-1}$;
\item $\varphi_{ij}\comp\varphi_{jk}\comp\varphi_{ki}=\alpha_{ijk}\cdot\id$.
\end{itemize}
We denote by $\Qcoh(X,\alpha)$ the abelian category of such
$\alpha$-twisted quasi-coherent sheaves on $X$. It is proved in
\cite[Proposition 2.1.3.3]{Li} that this definition coincides with the
alternative one in terms of quasi-coherent sheaves on the gerbe
$\mathcal{X}\to X$ on $X$ associated to $\alpha$.

\begin{prop}
If $X$ is a scheme and $\alpha\in\Br(X)$, then $\Qcoh(X,\alpha)$ is a Grothendieck abelian category.
\end{prop}

\begin{proof}
The same argument used in the proof of \cite[Proposition 3.2]{A}
(where $X$ is assumed to be quasi-compact and quasi-separated) works
in this greater generality.\footnote{We thank Benjamin Antieau for
  pointing this out.} Indeed, denoting by $\mathcal{X}\to X$ the
gerbe associated to $\alpha$, one just needs to know that
$\Qcoh(\mathcal{X})$ is a Grothendieck category, which is true because
in any case $\mathcal{X}$ is an algebraic stack.
\end{proof}

It is then clear from \autoref{thm:main1} that we can deduce the following.

\begin{cor}\label{cor:geo2}
If $X$ is a scheme and $\alpha\in\Br(X)$, then the triangulated category $\D(\Qcoh(X,\alpha))$ has a unique enhancement.
\end{cor}

\section{The case of the category of compact objects}\label{sect:crit2}

In this section we prove \autoref{thm:crit2}. This needs some preparation. In particular, using the arguments in \autoref{subsect:Grothabstr}, we construct an exact equivalence $\dgD(\cA)/\cL\iso\D(\cG)$, for some localizing subcategory $\cL$ of $\dgD(\cA)$ and reduce to the criterion for uniqueness due to Lunts and Orlov (see \cite[Theorem 2]{LO}). Verifying that the assumptions of Lunts--Orlov's result are satisfied is the main and most delicate task of this section.

\subsection{The first reduction}\label{subsect:firstred}

If $\cG$ is a Grothendieck category and $\cA$ is a small set of generators of $\cG$ which we think of as a full subcategory of $\cG$, we know from \autoref{subsect:Grothabstr} that there is a pair of adjoint functors
\[
\fT\colon\cM\to\cG\qquad\fS\colon\cG\to\cM
\]
where $\cM:=\Mod{\cA}$.

As already explained in the proof of \autoref{thm:main1}, the quotient $\D(\cM)/\D_\cN(\cM)$ (where $\cN:=\ker\fT$) is naturally equivalent to $\dgD(\cA)/\cL$. For this, we think of $\cA$ as a dg category sitting in degree $0$ and we take $\cL$ to be the localizing subcategory corresponding to $\D_\cN(\cM)$ under the natural equivalence $\D(\cM)\iso\dgD(\cA)$. Moreover, there is an exact equivalence $\D(\cG)\to\dgD(\cA)/\cL$, such that $A$ in $\cA$, seen as a subcategory of $\cG$, is mapped to $\fQ(\Yon(A))$, where $\fQ\colon\dgD(\cA)\to\dgD(\cA)/\cL$ is the quotient functor. As a consequence,
\[
\dgD(\cA)/\cL(\fQ(\Yon(A_1)),\fQ(\Yon(A_2))[i])=0,
\]
for all $A_1,A_2\in\cA$ and all integers $i<0$.

Consider now the following result.

\begin{thm}[\cite{LO}, Theorem 2]\label{thm:LO}
Let $\cA$ be a small category and let $\cL$ be a localizing subcategory of $\dgD(\cA)$ such that:
\begin{itemize}
\item[{\rm (a)}] $\cL^c=\cL\cap\dgD(\cA)^c$ and $\cL^c$ satisfies {\rm (G1)} in $\cL$;
\item[{\rm (b)}] $\dgD(\cA)/\cL(\fQ(\Yon(A_1)),\fQ(\Yon(A_2))[i])=0$, for all $A_1,A_2\in\cA$ and all integers $i<0$.
\end{itemize}
Then $(\dgD(\cA)/\cL)^c$ has a unique enhancement.
\end{thm}

By the discussion above, (b) is verified. If we could prove that the same is true for (a), then this theorem would immediately imply that $\D(\cG)^c$ has a unique enhancement. Thus, in order to prove \autoref{thm:crit2}, it is enough to show that the assumptions (1)--(4) in the statement imply that (a) holds. This delicate check will be the content of the next section.

\subsection{Verifying assumption (a) in \autoref{thm:LO}}\label{subsect:assumpt}

In the following we need to know the precise definition of $\fT\colon\cM\to\cG$. 
To this purpose, we first fix some notation. For $M\in\cM$, let
$\comma{M}$ be the comma category whose objects are pairs $(A,a)$ with
$A\in\cA$ and $a\in\cM(\YA{A},M)$, and whose morphisms are
given by
\[\comma{M}((A',a'),(A,a)):=
\{f\in\cA(A',A)\st a'=a\comp\YA{f}\}.\]
Observe that, by Yoneda's lemma, $\cM(\YA{A},M)$ can be
identified with $M(A)$ and that, in this way, the above equality
$a'=a\comp\YA{f}$ becomes $a'=M(f)(a)$; in what follows we will freely
use these identifications. Denoting by $\fgt{M}\colon\comma{M}\to\cA$
the forgetful functor, it is well known (see, for instance,
\cite[Section III.7]{M}) that
\[
M\iso\dlim(\comma{M}\mor{\fgt{M}}\cA\mor{\Yon}\cM),
\]
where the colimit is taken over the composition $\Yon\circ\fgt{M}$.
Since $\fT$ (being a left adjoint) preserves colimits and $\fT\comp\Yon$ is isomorphic to the inclusion $\cA\mono\cG$ (see \autoref{rmk:conseq}), we obtain
\[
\fT(M)\iso\dlim(\comma{M}\mor{\fgt{M}}\cA\mono\cG).
\]
More explicitly, consider the objects of $\cG$
\[\limtar{M}:=\Plus_{(A,a)\in\comma{M}}A,\qquad
\limsrc{M}:=\Plus_{(f\colon(A',a')\to(A,a))\in\Mor\comma{M}}A',\]
and denote by $\inc{(A,a)}\colon A\mono\limtar{M}$ (for every object
$(A,a)$ of $\comma{M}$) and $\inc{f}\colon A'\mono\limsrc{M}$ (for
every morphism $f\colon(A',a')\to(A,a)$ of $\comma{M}$) the
natural morphisms. Then, by (the dual version of) \cite[Theorem 2, p.\ 113]{M}, we have:

\begin{lem}\label{coker}
There is a natural isomorphism
\[
\fT(M)\iso\cok(\limmap{M}\colon\limsrc{M}\to\limtar{M}),
\]
where, for every morphism $f\colon(A',a')\to(A,a)$ of
$\comma{M}$,
\begin{equation}\label{limmap}
\limmap{M}\comp\inc{f}:=\inc{(A',a')}-\inc{(A,a)}\comp f.
\end{equation}
\end{lem}

Let us now move to the core of the proof that assumption (a) in \autoref{thm:LO} holds in our situation. As explained in the introduction, we do not expect it to hold true in general. This is the reason why we need the further assumptions (1)--(4) in \autoref{thm:crit2}. For the convenience of the reader, we list them again here:
\begin{enumerate}
\item\label{sum} $\cA$ is closed under finite coproducts;
\item\label{noeth} Every object of $\cA$ is noetherian in $\cG$;
\item\label{ker} If $f\colon A'\epi A$ is an epimorphism of $\cG$ with
$A,A'\in\cA$, then $\ker f\in\cA$;
\item\label{Ext} For every $A\in\cA$ there exists $N(A)>0$ such that
$\D(\cG)(A,\sh[N(A)]{A'})=0$ for every $A'\in\cA$.
\end{enumerate}

\begin{remark}\label{finsum}
If $f\colon\Plus_{i\in I}C_i\to C$ (with $I$ a small set) is a morphism in $\cG$ and
$B$ is a noetherian subobject of $C$ such that $B\subseteq\im f$,
then there exists a finite subset $I'$ of $I$ such that $B\subseteq
f(\Plus_{i\in I'}C_i)$ (for otherwise we could find elements
$i_1,i_2,\dots$ in $I$ such that $f(\Plus_{j=1}^nC_{i_j})\cap B$
for $n>0$ form a strictly increasing sequence of subobjects of $B$).
\end{remark}

\begin{lem}\label{epicomp}
Assume that conditions \eqref{sum} and \eqref{noeth} above are satisfied. If
$f\colon C\epi A$ is an epimorphism of $\cG$ with $A\in\cA$, then
there exists a morphism $g\colon A'\to C$ with $A'\in\cA$ such that
$f\comp g\colon A'\epi A$ is again an epimorphism of $\cG$.
\end{lem}

\begin{proof}
Given $f$ as in the statement, there exist a small set $I$ and an epimorphism
$g'\colon\Plus_{i\in I}A_i\epi C$ (so that $f\comp g'$ is also an
epimorphism) with $A_i\in\cA$ for every $i\in I$ (because the objects
of $\cA$ form a small set of generators of $\cG$). As $A$ is noetherian in
$\cG$ by condition \eqref{noeth}, \autoref{finsum} implies that
there exists a finite subset $I'$ of $I$ such that, setting
$A':=\Plus_{i\in I'}A_i$ (which is an object of $\cA$ thanks to
condition \eqref{sum}) and $g:=g'\rest{A'}$, the composition $f\comp
g\colon A'\epi A$ is an epimorphism of $\cG$.
\end{proof}

\begin{prop}\label{torsion}
If conditions \eqref{sum} and \eqref{noeth} are satisfied, then $\cN$
coincides with the full subcategory $\cN'$ of $\cM$ having as objects
those $M\in\cM$ satisfying the following property: for every object
$(A,a)$ of $\comma{M}$ there exists an epimorphism $f\colon A'\epi A$
of $\cG$ with $A'\in\cA$ such that $a\comp\YA{f}=0$.
\end{prop}

\begin{proof}
Given $M\in\cN'$, we have to prove that $\fT(M)\iso0$. By
\autoref{coker}, this is true if and only if $\limmap{M}$ is an
epimorphism. So, given a morphism $g\colon\limtar{M}\to C$ in $\cG$
such that $g\comp\limmap{M}=0$, we need to show that $g=0$. Now, if
$g$ is given by morphisms $g_{(A,a)}\colon A\to C$ for every
$(A,a)\in\comma{M}$, then $g\comp\limmap{M}=0$ is equivalent, by
\eqref{limmap}, to $g_{(A',a')}=g_{(A,a)}\comp f$ for every
morphism $f\colon(A',a')\to(A,a)$ of $\comma{M}$. Since
$M\in\cN'$, for every $(A,a)\in\comma{M}$ there exists an epimorphism
$f\colon A'\epi A$ of $\cG$ with $A'\in\cA$ such that
$a\comp\YA{f}=0$. Then $f,0\colon(A',0)\to(A,a)$ are morphisms of
$\comma{M}$, whence
\[g_{(A,a)}\comp f=g_{(A',0)}=g_{(A,a)}\comp0=0.\]
As $f$ is an epimorphism, we conclude that $g_{(A,a)}=0$, thus proving
that $g=0$.

Conversely, assume that $N\in\cN$, and fix an object $(A,a)$ of
$\comma{N}$. Since $\limmap{N}$ is an epimorphism (again by
\autoref{coker}) and $A$ is a noetherian object of $\cG$, by
\autoref{finsum} we can find a
finite number of distinct morphisms of $\comma{N}$, say
$f_i\colon(A'_i,a'_i)\to(A_i,a_i)$ for $i=1,\dots,n$, such that,
setting
\[
A'_0:=\Plus_{i=1}^nA'_i\subset\limsrc{N},
\]
we have
$\inc{(A,a)}(A)\subseteq\limmap{N}(A'_0)$. Moreover,
\[
\limmap{N}(A'_0)\subseteq A_0:=\Plus_{(A',a')\in I}A'\subset
\limtar{N},
\]
where $I$ is the (finite) subset of the objects of $\comma{N}$
consisting of those $(A',a')$ which are equal to $(A'_i,a'_i)$ or
$(A_i,a_i)$ for some $i=1,\dots,n$. Note that $A_0,A'_0\in\cA$ by
condition \eqref{sum}. In the cartesian diagram in $\cG$
\[\xymatrix{B \ar[rr]^-{f'} \ar[d]_{g'} & & A \ar[d]^-{\inc{(A,a)}} \\
A'_0 \ar[rr]_-{\limmap{N}\rest{A'_0}} & & A_0}\]
the morphism $f'$ is an epimorphism because
$\inc{(A,a)}(A)\subseteq\limmap{N}(A'_0)$. So, by \autoref{epicomp},
there exists a morphism $k\colon A'\to B$ with $A'\in\cA$ such that
$f:=f'\comp k\colon A'\epi A$ is an epimorphism of $\cG$. Setting also
$g:=g'\comp k\colon A'\to A'_0$ and denoting by
\[
a_0\colon\YA{A_0}\iso
\Plus_{(A',a')\in I}\YA{A'}\to N
\]
 the morphism of $\cM$ whose
components are given by $a'$ for every $(A',a')\in I$, the diagram
\[\xymatrix{\YA{A'} \ar[rr]^-{\YA{f}} \ar[d]_{\YA{g}} & & \YA{A}
\ar[d]_-{\YA{\inc{(A,a)}}} \ar[rr]^{a} & & N \\
\YA{A'_0} \ar[rr]_-{\YA{\limmap{N}\rest{A'_0}}} & &
\YA{A_0} \ar[urr]_-{a_0}}\]
commutes in $\cM$. As $\YA{A'_0}\iso\Plus_{i=1}^n\YA{A'_i}$ and
\[a_0\comp\YA{\limmap{N}\rest{A'_0}\comp\inc{f_i}}=a'_i-a_i\comp\YA{f_i}=0\]
for every $i=1,\dots,n$ (by \eqref{limmap} and by definition of
morphism in $\comma{N}$), we obtain that
$a_0\comp\YA{\limmap{N}\rest{A'_0}}=0$. This clearly implies that
$a\comp\YA{f}=0$, which proves that $N\in\cN'$.
\end{proof}

\begin{thm}\label{cptgen}
Assume that conditions \eqref{sum}, \eqref{noeth}, \eqref{ker} and
\eqref{Ext} are satisfied. Then $\D_{\cN}(\cM)\cap\D(\cM)^c$ satisfies {\rm (G1)} in $\D_{\cN}(\cM)$.
\end{thm}

\begin{proof}
In the triangulated category $\D_{\cN}(\cM)$, consider an object
\[
M=(\cdots\to M^0\mor{m^0}M^1\to\cdots)
\]
such that $M\niso0$. We must find a morphism $0\ne x\colon P\to M$ with
$P$ in $\D_{\cN}(\cM)\cap\D(\cM)^c$. Setting $N^i:=H^i(M)$, by definition
$N^i\in\cN$ for every $i\in\ZZ$ and $N^i\ne0$ for at least one $i$.
Without loss of generality we can assume that $N^0\ne0$, hence there
exists $(A^0,\bar{a}^0)\in\comma{N^0}$ with $\bar{a}^0\ne0$.

We claim that we can find a complex
\[
A=(0\to A^{-n}\mor{d^{-n}}\cdots\mor{d^{-1}}A^0\to0)
\]
of $\cA\subseteq\cG$ with $n=N(A^0)$ such that $H^i(A)=0$ for every $i\ne-n$. Here $N(A^0)$ is the integer whose existence is prescribed by \eqref{Ext} applied to $A^0$.
Furthermore, we will show that there is a morphism $a\colon\YA{A}\to M$ of complexes of $\cM$ (with
components $a^i\colon\YA{A^i}\to M^i$) such that, denoting by
$p^i\colon\ker m^i\epi N^i$ the natural projection morphism for every
$i\in\ZZ$, $p^0\comp a^0=\bar{a}^0$. Notice that, since $m^0\comp
a^0=0$, we can regard $a^0$ as a morphism $\YA{A^0}\to\ker
m^0$. Moreover, observe that for such a complex $A$ the objects $K^i:=\ker d^i$
of $\cG$ are actually in $\cA$. Indeed, this is clear for $i\ge0$ or
$i<-n$, whereas for $-n\le i<0$ there is a short exact sequence
\[
\xymatrix{
0\ar[r]&K^i\ar[r]^-{j^i}&A^i\ar[r]&K^{i+1}\ar[r]&0
}
\]
in $\cG$ (because $H^{i+1}(A)=0$), hence
one can prove that $K^i\in\cA$ by descending induction on $i$ using
condition \eqref{ker}.

In order to prove the claim, we define the morphisms $a^i$ and $d^i$
again by descending induction on $i$. For $i=0$, we can find
$a^0\colon\YA{A^0}\to\ker m^0 \subseteq M^0$ such that $p^0\comp
a^0=\bar{a}^0$ because $p^0$ is an epimorphism in $\cM$. As for the
inductive step, assume that $-n\le i<0$ and that suitable $a^{i'}$ and
$d^{i'}$ have already been defined for $i'>i$. There exists (unique)
$k^{i+1}\colon\YA{K^{i+1}}\to\ker m^{i+1}$ such that the diagram
\[\xymatrix{\YA{K^{i+1}} \ar[d]^{k^{i+1}}
\ar[rr]^-{\YA{j^{i+1}}} & & \YA{A^{i+1}}
\ar[d]^{a^{i+1}} \ar[rr]^-{\YA{d^{i+1}}} & & \YA{A^{i+2}} \ar[d]^{a^{i+2}} \\
\ker m^{i+1} \ar@{^{(}->}[rr] & & M^{i+1} \ar[rr]_-{m^{i+1}} & & M^{i+2}}\]
commutes (because $d^{i+1}\comp j^{i+1}=0$ and the square on the right
commutes by induction). Consider the object $(K^{i+1},p^{i+1}\comp
k^{i+1})$ of $\comma{N^{i+1}}$. Since $N^{i+1}\in\cN$, by
\autoref{torsion} there exists an epimorphism $q^i\colon A^i\epi K^{i+1}$
such that $p^{i+1}\comp k^{i+1}\comp\YA{q^i}=0$. So
\[
k^{i+1}\comp\YA{q^i}\colon\YA{A^i}\to\ker m^{i+1}
\]
factors through
$\im m^i\mono\ker m^{i+1}$, and there exists a morphism
$a^i$ such that the diagram
\[\xymatrix{\YA{A^i} \ar[d]_{a^i} \ar[drr] \ar[rrrr]^-{\YA{q^i}} & & & &
\YA{K^{i+1}} \ar[d]^{k^{i+1}} \ar[rr]^-{\YA{j^{i+1}}} & & \YA{A^{i+1}}
\ar[d]^{a^{i+1}} \\
M^i \ar@(dr,dl)[rrrrrr]_-{m^i} \ar@{->>}[rr] & & \im m^i
\ar@{^{(}->}[rr] & & \ker m^{i+1} \ar@{^{(}->}[rr] & & M^{i+1} \\
}\]
commutes. Then, setting $d^i:=j^{i+1}\comp q^i$, we clearly have
\[
H^{i+1}(A)=0\quad\text{and}\quad a^{i+1}\comp\YA{d^i}=m^i\comp a^i,
\]
thus completing the proof of the inductive step.

As $A\iso\sh[n]{K^{-n}}$ in $\D(\cG)$ and $K^{-n}\in\cA$, the natural
morphism of complexes $l\colon A^0\to A$ defined by $l^0=\id_{A_0}$ is
$0$ in $\D(\cG)(A^0,A)\iso\D(\cG)(A^0,\sh[n]{K^{-n}})$
by condition \eqref{Ext}. Thus we can find a complex $C$ of $\cG$ and
a quasi-isomorphism $r\colon C\to A^0$ such that $l\comp r\htp0$,
where $\htp$ denotes homotopy of morphisms of complexes. As $H^i(C)$
is isomorphic to an object of $\cA$ for every $i\in\ZZ$ and is $0$ for
$i>0$, there exists a quasi-isomorphism $s:B\to C$ with $B^i\in\cA$
for every $i\in\ZZ$ and $B^i=0$ for $i>0$: this follows for instance
from \cite[Lemma 1.9.5]{TT} (applied with $F$ the inclusion of $\cA$
in $\cG$ and $\mathcal{C}$ the full subcategory of the category of
complexes in $\cG$ having as objects the complexes whose cohomologies
are bounded above and isomorphic to objects of $\cA$), whose key
condition 1.9.5.1 is satisfied due to \autoref{epicomp}. Then
$t:=r\comp s\colon B\to A^0$ is also a quasi-isomorphism and $l\comp
t\htp0$. It is straightforward to check that $t$ factors through a
quasi-isomorphism $\tilde{t}\colon\tilde{B}:=\tau_{\ge-n}(B)\to A^0$ and
that $l\comp\tilde{t}\htp0$, too. Indeed, the same maps $B^i\to A^{i-1}$, which provide the homotopy $l\comp
t\htp0$ and are necessarily zero for $i\leq -n$, yield the desired homotopy $l\comp\tilde{t}\htp0$.

Hence, denoting by $\tilde{A}$ the
mapping cone of $\tilde{t}$ and by $u\colon A^0\to\tilde{A}$ the
natural inclusion, there exists a morphism of complexes
$f\colon\tilde{A}\to A$ such that $f\comp u\htp l$. It is easy to prove, applying \eqref{ker} and the same argument used above to show that $K^i\in\cA$, that $\tilde{B}^i\in\cA$, for every integer $i$. It follows from \eqref{sum} that the same is true for $\tilde{A}^i$. 

Now we can take $P:=\YA{\tilde{A}}$ and $x:=a\comp\YA{f}\colon P\to M$
(or, better, its image in $\D(\cM)$). Indeed, $x\comp\YA{u}\htp
a\comp\YA{l}=a^0$, which implies
\[
H^0(x\comp\YA{u})=H^0(a^0)=\bar{a}^0\ne0.
\]
Therefore
$x\comp\YA{u}\ne0$, whence $x\ne0$ in $\D(\cM)$. Moreover,
$P\in\D(\cM)^c$ because $\D(\cM)^c$ is a triangulated subcategory of
$\D(\cM)$ containing the image of $\Yon$ (see \autoref{gen}), and $\tilde{A}$ is a bounded complex
of objects of $\cA$. Finally, we have $\fT(P)\iso\tilde{A}$ by \autoref{rmk:conseq}. Remembering that $\fT$
is exact and observing that $\tilde{A}$ is an acyclic complex (being
the mapping cone of the quasi-isomorphism $\tilde{t}$), we conclude
that
\[
\fT(H^i(P))\iso H^i(\fT(P))\iso H^i(\tilde{A})=0
\]
for every $i\in\ZZ$, which means that $P\in\D_{\cN}(\cM)$.
\end{proof}

An easy application of the above result is the following.

\begin{cor}\label{Neeman}
Assume that conditions \eqref{sum}, \eqref{noeth}, \eqref{ker} and
\eqref{Ext} are satisfied. Then	
\begin{enumerate}
\item[(i)] $\D_{\cN}(\cM)^c=\D_{\cN}(\cM)\cap\D(\cM)^c$;
\item[(ii)] The quotient functor $\D(\cM)\to\D(\cM)/\D_{\cN}(\cM)$ sends $\D(\cM)^c$ to $(\D(\cM)/\D_{\cN}(\cM))^c$;
\item[(iii)] The induced functor $\D(\cM)^c/\D_{\cN}(\cM)^c\to (\D(\cM)/\D_{\cN}(\cM))^c$ is fully faithful and identifies $(\D(\cM)/\D_{\cN}(\cM))^c$ with the idempotent completion of $\D(\cM)^c/\D_{\cN}(\cM)^c$.
\end{enumerate}
\end{cor}

\begin{proof}
The triangulated category $\D(\cM)$ is compactly generated by \autoref{gen}, and so the isomorphism classes of objects in $\D(\cM)^c$ form a small set (for this use, for example, \cite[Lemma 5]{K1}). Thus we can choose a small set $\cS$ of representatives of the isomorphism classes of objects in $\D_\cN(\cM)\cap\D(\cM)^c$, and \autoref{cptgen} clearly implies that $\cS$ satisfies (G1) in $\D_\cN(\cM)$. Since $\D(\cM)/\D_\cN(\cM)\iso\D(\cG)$ is well generated (see \autoref{ex:alphacompgen}), the localizing subcategory $\D_\cN(\cM)$ is well generated as well (see, for example, \cite[Theorem 7.4.1]{K2}). Hence, by \autoref{prop:PortaG1}, the category $\D_\cN(\cM)$ is generated by $\cS$. Now we just apply \autoref{thm:NeemanComp}.
\end{proof}

\autoref{cptgen} and part (i) of \autoref{Neeman} imply that the assumption (a) of \autoref{thm:LO} is satisfied, in our specific situation (i.e.\ when (1)--(4) are satisfied). Hence the proof of \autoref{thm:crit2} is complete.

\begin{remark}\label{rmk:strun}
It should be noted that, under the same assumptions (1)--(4) in \autoref{thm:crit2}, one can actually prove that the triangulated category $\D(\cG)^c$ has a semi-strongly unique enhancement. This result follows again from \autoref{cptgen} and part (i) of \autoref{Neeman} using \cite[Theorem 6.4]{LO}, rather than \autoref{thm:LO}.
\end{remark}

\subsection{The geometric examples}\label{subsect:compex}

In this section we describe an easy geometric application of \autoref{thm:crit2} in the case of perfect complexes on some algebraic stacks. For this we need to recall some definitions.

Let $R$ be a commutative ring. A complex $P\in\D(\Mod{R})$ is \emph{perfect} if it is quasi-isomorphic to a bounded complex of projective $R$-modules of finite presentation. Following \cite{HR}, if $X$ is an algebraic stack, a complex $P\in\Dq(X)$ is \emph{perfect} if for any smooth morphism $\spec(R)\to X$, where $R$ is a commutative ring, the complex of $R$-modules $\rd\Gamma(\spec(R),P\rest{\spec(R)})$ is perfect. We denote by $\Dp(X)$ the full subcategory of $\Dq(X)$ consisting of perfect complexes.

A quasi-compact and quasi-separated algebraic stack $X$ is \emph{concentrated} if $\Dp(X)\subseteq\Dq(X)^c$. On the other hand, the other inclusion $\Dq(X)^c\subseteq\Dp(X)$ holds as well, by \cite[Lemma 4.4]{HR}. Moreover, we already observed in the proof of \autoref{cor:geo1} that, under the same assumptions, the natural functor $\D(\Qcoh(X))\to\Dq(X)$ is an exact equivalence.

Summing up, if $X$ is a concentrated algebraic stack with quasi-finite affine diagonal, then there is a natural exact equivalence
\begin{equation}\label{perfcpt}
\Dp(X)\iso\D(\Qcoh(X))^c.
\end{equation}

When a stack $X$ has the property that $\Qcoh(X)$ is generated, as a Grothendieck category, by a small set of objects contained in $\Coh(X)\cap\Dp(X)$, we say that \emph{$X$ has enough perfect coherent sheaves}.

\begin{ex}\label{ex:enough}
Suppose that a scheme $X$ has enough locally free sheaves, according to the definition given in \autoref{rmk:schemes1}.  This yields a small set of
generators of $\Qcoh(X)$ contained in $\Coh(X)\cap\Dp(X)$. Indeed, we can take a set of representatives for the isomorphism classes of locally free sheaves, as every sheaf in $\Qcoh(X)$ is a filtered colimit of finitely presented $\ko_X$-modules (see \cite[9.4.9]{EGA1}). Hence a scheme with enough locally free sheaves has enough perfect coherent sheaves as well.
\end{ex}

As an application of \autoref{thm:crit2}, we get the following.

\begin{prop}\label{prop:geocomp}
Let $X$ be a noetherian concentrated algebraic stack with quasi-finite affine diagonal and enough perfect coherent sheaves. Then $\Dp(X)$ has a unique enhancement.
\end{prop}

\begin{proof}
Consider the isomorphism classes of objects in $\Coh(X)\cap\Dp(X)$. It is clear that they form a small set. Define then $\cA$ to be the full subcategory of $\Qcoh(X)$ whose set of objects is obtained by taking a representative in each isomorphism class of objects in $\Coh(X)\cap\Dp(X)$. Since, by assumption, a subset of $\Coh(X)\cap\Dp(X)$ generates $\Qcoh(X)$, $\cA$ does the same.

Let us now observe that $\cA$ satisfies (1)--(4) in \autoref{thm:crit2}. Indeed, (1) is obvious and (2) holds true because $X$ is noetherian. To prove (3), observe that the kernel is defined in $\Coh(X)$ up to isomorphism and moreover, the kernel of an epimorphism $A\epi A'$ in $\cA$ is isomorphic to the shift of the cone of $f$ in $\Dp(X)$. Hence it is (up to isomorphism) an object in $\cA$. Finally, since $X$ is concentrated, (4) is verified as well. Indeed, in view of \cite[Remark 4.6]{HR}, a concentrated stack has finite homological dimension and then (4) follows from the fact that the objects of $\cA$ are in $\Coh(X)\cap\Dp(X)$ rather than just in $\Coh(X)$.

At this point, the result follows directly from \autoref{thm:crit2} and \eqref{perfcpt}.
\end{proof}

As a direct consequence, we get the following.

\begin{cor}\label{cor:geocomp}
If $X$ is a noetherian scheme with enough locally free sheaves, then $\Dp(X)$ has a unique enhancement.
\end{cor}

\begin{proof}
A scheme that is noetherian is concentrated (see \cite[Theorem 3.1.1]{BB}). Moreover, by \cite[Proposition 1.3]{Tot}, any noetherian scheme with enough locally free sheaves is semi-separated. By \autoref{ex:enough} and \autoref{prop:geocomp}, the result is then clear.
\end{proof}

\section{Applications}\label{sect:applications}

In this section we discuss two easy applications of the circle of ideas concerning the uniqueness of enhancements for the category of perfect complexes. The first one is about a uniqueness result for the enhancements of the bounded derived category of coherent sheaves. The second one concerns some basic questions related to exact functors between the categories of perfect complexes or the complexes of quasi-coherent sheaves.

\subsection{The bounded derived category of coherent sheaves}\label{subsect:Db}

Assume again that $X$ is a noetherian scheme with enough locally free sheaves (which is again automatically semi-separated in view of \cite[Proposition 1.3]{Tot}). Let $\cA$ be a full subcategory of $\Qcoh(X)$ whose objects are obtained by picking a representative in each isomorphism class of the objects in $\Coh(X)\cap\Dp(X)$. As we observed in the proof of \autoref{cor:geocomp}, $\cA$ is a small set of generators of $\Qcoh(X)$. Hence, we can apply the discussion in \autoref{subsect:Grothabstr}, getting a natural exact equivalence
\begin{equation}\label{eqn:eqcat}
\D(\Qcoh(X))\iso\dgD(\cA)/\cL,
\end{equation}
where $\cL$ is an explicit localizing subcategory of $\dgD(\cA)$. Remember that, under this equivalence, every object $A\in\cA$ is mapped to $\fQ(\Yon(A))$, where, as usual, $\fQ\colon\dgD(\cA)\to\dgD(\cA)/\cL$ denotes the quotient functor (see the discussion in \autoref{subsect:firstred} about this point). Since $\cA\subseteq\D(\Qcoh(X))^c$, in view of \eqref{perfcpt}, it follows from \cite[Remark 1.20]{LO} that $\cS:=\fQ\comp\Yon(\cA)$ is a small set of compact generators of $\dgD(\cA)/\cL$.

Following \cite[Section 8]{LO}, we say that an object $B$ in $\dgD(\cA)/\cL$ is \emph{compactly approximated by the objects in $\cS$} if
\begin{enumerate}
\item There is $m\in\ZZ$ such that, for any $S\in\cS$, we have $\dgD(\cA)/\cL(S,B[i])\iso0$ when $i<m$;
\item For any $k\in\ZZ$, there are $P_k$ in $(\dgD(\cA)/\cL)^c$ and a morphism $f_k\colon P_k \to B$ such
that, for every $S\in\cS$, the canonical map
\[
\dgD(\cA)/\cL\left(S,P_k[i]\right)\longrightarrow \dgD(\cA)/\cL\left(S,B[i]\right)
\]
is an isomorphism when $i\geq k$.
\end{enumerate}
We denote by $(\dgD(\cA)/\cL)^{ca}$ the full subcategory of $\dgD(\cA)/\cL$ consisting of the objects which are compactly approximated by $\cS$. It must be noted that this definition actually depends on the choice of $\cS$.

Denote by $\Db(X)$ the bounded derived category of the abelian category $\Coh(X)$ of coherent sheaves on $X$. We have that the equivalence \eqref{eqn:eqcat} induces an exact equivalence
\begin{equation}\label{Dbca}
\Db(X)\iso(\dgD(\cA)/\cL)^{ca}.
\end{equation}

Indeed, the proof of \cite[Proposition 8.9]{LO} can be repeated line by line and the same argument applies in our setting. The only delicate issue is that \cite[Lemma 8.10]{LO} has to be replaced by the following statement: Let $X$ be a scheme as above and let $E$ be the image under \eqref{eqn:eqcat} of an object in $\cS$. Then there exists an integer $N(E)$ such that, for all $k\geq N(E)$ and all quasi-coherent sheaves $F$, we have $\Ext^k(E,F)=0$. This follows easily from the fact that $E$ is, by definition, a coherent sheaf in $\Dp(X)$. Thus the results in Sections B.11 and B.12 of \cite{TT} apply.

Consider now the following result.

\begin{thm}[\cite{LO}, Theorem 8.8]\label{thm:LO2}
Let $\cA$ be a small category and let $\cL$ be a localizing subcategory of $\dgD(\cA)$ such that:
\begin{itemize}
\item[{\rm (a)}] $\cL^c=\cL\cap\dgD(\cA)^c$ and $\cL^c$ satisfies {\rm (G1)} in $\cL$;
\item[{\rm (b)}] $\dgD(\cA)/\cL(\fQ(\Yon(A_1)),\fQ(\Yon(A_2))[i])=0$, for all $A_1,A_2\in\cA$ and all integers $i<0$.
\end{itemize}
Then $(\dgD(\cA)/\cL)^{ca}$ has a unique enhancement.
\end{thm}

This has the following easy consequence.

\begin{cor}\label{cor:geocomp2}
If $X$ is a noetherian scheme with enough locally free sheaves, then $\Db(X)$ has a unique enhancement.
\end{cor}

\begin{proof}
The proofs of \autoref{prop:geocomp} and \autoref{cor:geocomp} actually show that, with these
assumptions on $X$ and our choice of $\cA$, hypotheses (a) and (b) of
\autoref{thm:LO} are satisfied. As they coincide with (a) and (b) in
\autoref{thm:LO2}, we conclude by \eqref{Dbca}.
\end{proof}

\subsection{Fourier--Mukai functors}\label{subsect:FM} Assume that $X_1$ and $X_2$ are noetherian schemes.
Given $\ke\in\D(\Qcoh(X_1\times X_2))$, we define the exact functor $\FM{\ke}\colon\D(\Qcoh(X_1))\to\D(\Qcoh(X_2))$ as
\begin{equation*}\label{eqn:FM}
\FM{\ke}(-):=\rd(p_2)_*(\ke\lotimes p_1^*(-)),
\end{equation*}
where $p_i\colon X_1\times X_2\to X_i$ is the natural projection.

\begin{definition}\label{def:FM}
An exact functor $\fF\colon\D(\Qcoh(X_1))\to\D(\Qcoh(X_2))$ ($\fG\colon\Dp(X_1)\to\Dp(X_2)$, respectively) is a \emph{Fourier--Mukai functor} (or of \emph{Fourier--Mukai type}) if there exists an object $\ke\in\D(\Qcoh(X_1\times X_2))$ and an isomorphism of exact functors $\fF\iso\FM{\ke}$ ($\fG\iso\FM{\ke}$, respectively). 
\end{definition}

These functors are ubiquitous in algebraic geometry (see \cite{CS3} for a survey on the subject) and for a long while it was believed by some people that all exact functors between $\Db(X_1)$ and $\Db(X_2)$, with $X_i$ a smooth projective scheme, had to be of Fourier--Mukai type. A beautiful counterexample by Rizzardo and Van den Bergh \cite{RvdB} showed this expectation to be false. Moreover, if $X_1$ and $X_2$ are not smooth projective it is not even clear if the celebrated result of Orlov \cite{Or} asserting that all exact equivalences between $\Db(X_1)$ and $\Db(X_2)$ are of Fourier--Mukai type holds true.

A much weaker question can be now formulated as follows. For two triangulated categories $\cT_1$ and $\cT_2$, we denote by $\Eq(\cT_1,\cT_2)$ the set of isomorphism classes of exact equivalences between $\cT_1$ and $\cT_2$. When $\cT_i$ is either $\D(\Qcoh(X_i))$ or $\Dp(X_i)$, for $X_i$ a noetherian scheme, we can further define the subset $\Eq[FM](\cT_1,\cT_2)$ consisting of equivalences of Fourier--Mukai type.

As an application of the results in the previous section, we get the following.

\begin{prop}\label{prop:FM}
Let $X_1$ and $X_2$ be noetherian schemes with enough locally free sheaves. Then $\Eq(\Dp(X_1),\Dp(X_2))\neq\emptyset$ if and only if $\Eq(\D(\Qcoh(X_1)),\D(\Qcoh(X_2)))\neq\emptyset$. Moreover, each of the two equivalent conditions implies $\Eq(\Db(X_1),\Db(X_2))\neq\emptyset$.
\end{prop}

\begin{proof}
In view of \eqref{perfcpt}, an exact equivalence $\D(\Qcoh(X_1))\to\D(\Qcoh(X_2))$ restricts to an exact equivalence $\Dp(X_1)\to\Dp(X_2)$, since the subcategories of compact objects are clearly preserved. Hence, $\Eq(\D(\Qcoh(X_1)),\D(\Qcoh(X_2)))\neq\emptyset$ implies that the same is true for the categories of perfect complexes.

On the other hand, assume that $\Eq(\Dp(X_1),\Dp(X_2))\neq\emptyset$. Denoting by $\Perf(X_i)$ a dg enhancement of $\Dp(X_i)$, for $i=1,2$, by \autoref{cor:geocomp} $\Perf(X_1)\iso\Perf(X_2)$ in $\Hqe$. This clearly implies that there is an exact equivalence between $\dgD(\Perf(X_1))$ and $\dgD(\Perf(X_2))$. By \cite[Proposition 1.16]{LO} (see also the proof of \cite[Corollary 9.13]{LO}), there is an exact equivalence between $\dgD(\Perf(X_i))$ and $\D(\Qcoh(X_i))$, for $i=1,2$. Thus $\Eq(\D(\Qcoh(X_1)),\D(\Qcoh(X_2)))\neq\emptyset$.

As for the last statement, assume (without loss of generality by the previous part) that there is $\fF$ in $\Eq(\D(\Qcoh(X_1)),\D(\Qcoh(X_2)))$. By \cite[Proposition 6.9]{R}, the functor $\fF$ sends the subcategory $\Db(\Qcoh(X_1))$ of cohomologically bounded complexes to $\Db(\Qcoh(X_2))$. By using the same argument as above, we see that $\fF$ induces an exact equivalence
\[
\Db(\Qcoh(X_1))^c\longrightarrow\Db(\Qcoh(X_2))^c.
\]
Then we conclude that $\Eq(\Db(X_1),\Db(X_2))\neq\emptyset$, since $\Db(\Qcoh(X_i))^c\iso\Db(X_i)$, for $i=1,2$, by \cite[Corollary 6.16]{R}.
\end{proof}

Notice that, if we assume further that $X_1\times X_2$ is noetherian and that any complex in $\Dp(X_i)$ is isomorphic to a bounded complex of vector bundles, then \cite[Corollary 8.12]{To} and \cite[Theorem 1.1]{LS} imply that
\[
\Eq(\Dp(X_1),\Dp(X_2))\neq\emptyset\text{ iff } \Eq[FM](\Dp(X_1),\Dp(X_2))\neq\emptyset
\]
and
\[
\Eq(\D(\Qcoh(X_1)),\D(\Qcoh(X_2)))\neq\emptyset\text{ iff }\Eq[FM](\D(\Qcoh(X_1)),\D(\Qcoh(X_2)))\neq\emptyset.
\]
Hence \autoref{prop:FM} can be reformulated in terms of the sets of Fourier--Mukai equivalences.

\begin{remark}\label{rmk:FMstr}
By using the observation in \autoref{rmk:strun} and the strategy in the proof of \cite[Corollary 9.12]{LO}, we can make the above remarks more precise, when dealing with perfect complexes. Indeed, pick $\fF\in\Eq(\Dp(X_1),\Dp(X_2))$, for $X_i$ noetherian with enough locally free sheaves and such that $X_1\times X_2$ is noetherian and any complex in $\Dp(X_i)$ is isomorphic to a bounded complex of vector bundles. Then there exists $\fG\in\Eq[FM](\Dp(X_1),\Dp(X_2))$ such that $\fF(C)\iso\fG(C)$, for any $C$ in $\Dp(X_1)$.
\end{remark}


\bigskip

{\small\noindent{\bf Acknowledgements.} Parts of this paper were
  written while P.S.\ was visiting the Institut de Math\'ematiques de
  Jussieu, Universit\'e de Paris 7. The warm hospitality and the
  financial support of this institution is gratefully
  acknowledged. The second author was benefitting from the financial
  support by the ``National Group for Algebraic and Geometric
  Structures, and their Applications'' (GNSAGA--INDAM). It is a great
  pleasure to thank Pieter Belmans, Valery Lunts, Dmitri Orlov and Marco Porta for very useful conversations and comments. We thank Pawel Sosna for carefully reading this paper in a preliminary version and for suggesting several improvements in the exposition. Various email exchanges with Henning Krause were crucial for this work. We warmly thank him for kindly answering all our questions. We are vey grateful to Olaf Schn\"urer who pointed out that all noetherian schemes with enough locally free sheaves are semi-separated. Our deepest thanks goes to the anonymous referee who pointed out a mistake in the previous version of the proof of \autoref{thm:main1} and to Benjamin Antieau who pointed out a gap in the previous version of Theorem C.}



\begin{thebibliography}{99}

\bibitem{A} B.\ Antieau, \emph{A reconstruction theorem for abelian categories of twisted sheaves}, J.\ reine angew.\ Math.\  {\bf 2016} (2016), 175--188.

\bibitem{BN}  M.\ B\"okstedt, A.\ Neeman, \emph{Homotopy limits in triangulated categories}, Comp.\ Math.\ {\bf 86} (1993), 209--234.

\bibitem{BLL} A.\ Bondal, M.\ Larsen, V.\ Lunts, \emph{Grothendieck ring of pretriangulated categories}, Int.\ Math.\ Res.\ Not.\ {\bf 29} (2004), 1461--1495.

\bibitem{BB} A.\ Bondal, M. Van den Bergh, \emph{Generators and representability of functors in commutative and noncommutative geometry}, Moscow
Math. J. {\bf 3} (2003), 1--36.

\bibitem{Bou} A.K.\ Bousfield, \emph{The localization of spectra with respect to homology}, Topology {\bf 18} (1979), 257--281.

\bibitem{Cal} A.\ C\u{a}ld\u{a}raru, \emph{Derived categories of twisted sheaves on Calabi--Yau
manifolds}, Ph.-D.\ thesis Cornell (2000).

\bibitem{CS3} A.\ Canonaco, P.\ Stellari, \emph{Fourier--Mukai functors: a survey},  EMS Ser.\ Congr.\ Rep., Eur.\ Math.\ Soc.\ (2013), 27--60.

\bibitem{CS} A.\ Canonaco, P.\ Stellari, \emph{Fourier-Mukai functors in the supported case}, Compositio Math.\ {\bf 150} (2014), 1349--1383.
	
\bibitem{CENT} F.\ Casta\~{n}o Iglesias, P.\ Enache, C.\ N\u{a}st\u{a}sescu, B.\ Torrecillas, \emph{Gabriel--Popescu type theorems and applications}, Bull.\ Sci.\ Math.\ {\bf 128} (2004), 323--332.

\bibitem{Dr} V.\ Drinfeld, \emph{DG quotients of DG categories}, J.\ Algebra {\bf 272} (2004), 643--691.

\bibitem{DS} D.\ Dugger, B.\ Shipley, \emph{A curious example of triangulated-equivalent model categories which are not Quillen equivalent},  Algebr.\ Geom.\ Topol.\ {\bf 9} (2009), 135--166. 

\bibitem{EGA1} A.\ Grothendieck, J.\ Dieudonn\'e, \emph{\'El\'ements de g\'eom\'etrie alg\'ebrique. I}, Inst.\ Hautes \'Etudes Sci.\ Publ.\ Math.\ {\bf 4} (1960), 228 pp.

\bibitem{HNR} J.\ Hall, A.\ Neeman, D.\ Rydh, \emph{One positive and two negative results for derived categories of algebraic stacks}, to appear in: J.\ Inst.\ Math.\ Jussieu, arXiv:1405.1888.

\bibitem{HR} J.\ Hall, D.\ Rydh, \emph{Perfect complexes on algebraic stacks}, Compositio Math. {\bf 153} (2017), 2318--2367.

\bibitem{KTel} B.\ Keller, \emph{A remark on the generalized smashing conjecture}, Manuscripta Math.\ {\bf 84} (1994), 193--198.

\bibitem{Ke1} B.\ Keller, \emph{Deriving DG categories}, Ann.\ Sci.\ \'{E}cole Norm.\ Sup.\ {\bf 27} (1995), 63--102.

\bibitem{K} B.\ Keller, \emph{On differential graded categories}, International Congress of Mathematicians Vol.\ II, Eur.\ Math.\ Soc., Z\"urich (2006), 151--190.

\bibitem{Ko} M. Kontsevich, \emph{Homological algebra of Mirror Symmetry}, in: Proceedings of the International Congress of Mathematicians (Zurich, 1994, ed. S.D. Chatterji), Birkhauser, Basel (1995), 120--139.

\bibitem{K0} H.\ Krause, \emph{Deriving Auslander's formula}, Documenta Math.\ {\bf 20} (2015), 669--688.

\bibitem{K1} H.\ Krause, \emph{On Neeman's well generated triangulated categories}, Documenta Math.\  {\bf 6} (2001), 119--125.

\bibitem{K2} H.\ Krause, \emph{Localization theory for triangulated categories}, in: Triangulated categories, London Math.\ Soc.\ Lecture Note Ser.\ {\bf 375}, Cambridge Univ.\ Press (2010), 161--235.

\bibitem{LMB} G.\ Laumon, L.\ Moret-Bailly, \emph{Champs alg\'{e}briques},
Ergebnisse der Mathematik und ihrer Grenzgebiete.\ 3.\ Folge.\ A Series of Modern Surveys in Mathematics {\bf 39}, Springer-Verlag (2000), xii+208.

\bibitem{Li} M.\ Lieblich, \emph{Moduli of twisted sheaves}, Duke Math.\ J.\ {\bf 138} (2007), 23--118.

\bibitem{LO} V.\ Lunts, D.\ Orlov, \emph{Uniqueness of enhancements for triangulated categories}, J.\ Amer.\ Math.\ Soc.\ {\bf 23} (2010), 853--908.

\bibitem{LS} V.\ Lunts, O.M.\ Schn\"urer, \emph{New enhancements of derived categories of coherent sheaves and applications}, J.\ Algebra {\bf 446} (2016), 203--274.

\bibitem{M} S.\ Mac Lane, \emph{Categories for the Working
  Mathematician}, second edition, G.T.M.\ {\bf 5}, Springer-Verlag (1998).

\bibitem{N3} A.\ Neeman, \emph{On the derived category of sheaves on a manifold}, Documenta Math.\ {\bf 6} (2001), 483--488.

\bibitem{N1} A.\ Neeman, \emph{The connection between the K-theory localization theorem of Thomason, Trobaugh and Yao and the smashing subcategories of Bousfield and Ravenel}, Ann.\ Sci.\ \'{E}cole Norm.\ Sup.\ {\bf 25} (1992), 547--566.

\bibitem{NTel} A.\ Neeman, \emph{The chromatic tower for $\D(R)$. With an appendix by Marcel B\"okstedt}, Topology {\bf 31} (1992), 519--532.

\bibitem{N2} A.\ Neeman, \emph{Triangulated categories}, Annals of Mathematics Studies {\bf 148}, Princeton University Press (2001), viii+449.

\bibitem{Or} D.\ Orlov, \emph{Equivalences of derived categories and K3 surfaces}, J.\ Math.\ Sci.\ {\bf 84} (1997), 1361--1381

\bibitem{P} M.\ Porta, \emph{The Popescu-Gabriel theorem for
triangulated categories}, Adv.\ Math.\ {\bf 225} (2010), 1669--1715.

\bibitem{Ra} D.C.\ Ravenel, \emph{Localization with respect to certain periodic homology theories}, Amer.\ J.\ Math.\ {\bf 105} (1984), 351--414.

\bibitem{RvdB} A.\ Rizzardo, M.\ Van den Bergh, \emph{An example of a non-Fourier--Mukai functor between derived categories of coherent sheaves}, arXiv:1410.4039.

\bibitem{R} R.\ Rouquier, \emph{Dimensions of triangulated categories}, J.\ K-theory {\bf 1} (2008), 193--258.

\bibitem{S} S.\ Schwede, \emph{Topological triangulated categories}, arXiv:1201.0899.
  
\bibitem{SP} The Stacks Project Authors, \emph{The Stacks Project}, {\tt http://stacks.math.columbia.edu/}.

\bibitem{TT} R.\ W.\ Thomason, T.\ Trobaugh, \emph{Higher Algebraic
K-Theory of Schemes and of Derived Categories}, The Grothendieck
Festschrift III, Birkh\"auser, Boston, Basel, Berlin, (1990),
247--436.

\bibitem{To} B.\ To\"en, \emph{The homotopy theory of dg-categories and derived Morita theory}, Invent.\ Math.\ {\bf 167} (2007), 615--667.

\bibitem{Tot} B.\ Totaro, \emph{The resolution property for schemes and stacks}, J.\ Reine Angew.\ Math.\ {\bf 577} (2004), 1--22.

\end{thebibliography}
\end{document}